\documentclass[a4paper]{amsart}
\usepackage{graphicx}
\usepackage{amssymb}
\usepackage{amsmath}
\usepackage{amsthm}
\usepackage{amscd}
\usepackage[all,2cell]{xy}
\UseAllTwocells \SilentMatrices
\newtheorem{thm}{Theorem}[section]

\newtheorem{cor}[thm]{Corollary}
\newtheorem{lem}[thm]{Lemma}
\newtheorem{exm}[thm]{Example}

\newtheorem{conj}[thm]{Conjecture}

\newtheorem{prop}[thm]{Proposition}
\theoremstyle{definition}
\newtheorem{defn}[thm]{Definition}
\theoremstyle{remark}

\numberwithin{equation}{section}

\begin{document}
\title[The $\mathbf{D}$-standard and $\mathbf{K}$-standard categories]{The $\mathbf{D}$-standard and $\mathbf{K}$-standard categories}
\author[Xiao-Wu Chen, Yu Ye] {Xiao-Wu Chen, Yu Ye$^*$}

\thanks{$^*$ The corresponding author}
%\thanks{}
\date{\today}

\thanks{E-mail:
xwchen$\symbol{64}$mail.ustc.edu.cn, yeyu$\symbol{64}$ustc.edu.cn}
%\thanks{this is NOT for distribution}

\maketitle

\dedicatory{}%
\commby{}%
%\begin{center}
%\end{center}

\begin{abstract}
We introduce the notions of a $\mathbf{D}$-standard abelian category and a $\mathbf{K}$-standard additive category. We prove that for a finite dimensional algebra $A$, its module category is $\mathbf{D}$-standard if and only if any derived autoequivalence on $A$ is standard, that is, isomorphic to the derived tensor functor by a two-sided tilting complex. We prove that if the subcategory of projective $A$-modules is $\mathbf{K}$-standard, then the module category is $\mathbf{D}$-standard. We provide new examples of $\mathbf{D}$-standard module categories.
\end{abstract}

\tableofcontents

\section{Introduction}

Let $k$ be a field, and let $A$ be a finite dimensional $k$-algebra. We denote by $A\mbox{-mod}$ the category of finite dimensional left  $A$-modules and by $\mathbf{D}^b(A\mbox{-mod})$ its bounded derived category. By a derived equivalence between two algebras $A$ and $B$, we mean a $k$-linear triangle equivalence $F\colon \mathbf{D}^b(A\mbox{-mod})\rightarrow \mathbf{D}^b(B\mbox{-mod})$. It is an open question in \cite{Ric} whether any derived equivalence is \emph{standard}, that is, isomorphic to the derived tensor functor by a two-sided tilting complex.

We mention that there exists a triangle functor between the bounded derived categories of module categories, which is not isomorphic to the derived tensor functor by any complex of bimodules; see \cite{RVan}.

The above open question is answered affirmatively  in the following cases: (1) $A$ is hereditary \cite{MY}; (2) $A$ is (anti-)Fano \cite{Mina}; (3) $A$ is triangular \cite{Chen}. We mention that their proofs rely on the work \cite{Or} and \cite{AR}.

In the present paper, we take a different approach. Recall from \cite{Ric} that for a given derived equivalence $F$, there is a standard equivalence $F'$ such that they act  identically on objects. This motivates the following notion: a triangle autoequivalence $G$ on $\mathbf{D}^b(A\mbox{-mod})$ is called a \emph{pseudo-identity},  provided that $G$ acts on objects by the identity and that the restriction of $G$ to stalk complexes  equals the identity functor. Roughly speaking, a pseudo-identity is very close to the genuine identity functor on $\mathbf{D}^b(A\mbox{-mod})$. Then any derived equivalence $F\colon\mathbf{D}^b(A\mbox{-mod})\rightarrow \mathbf{D}^b(B\mbox{-mod})$ allows a factorization $F\simeq F'G$ with $G$ a pseudo-identity on $\mathbf{D}^b(A\mbox{-mod})$ and $F'$ a standard equivalence; moreover, such a factorization is unique; see Proposition \ref{prop:factor}.

We say that the module category $A\mbox{-mod}$ is \emph{$\mathbf{D}$-standard} if any pseudo-identity on $\mathbf{D}^b(A\mbox{-mod})$ is isomorphic to the identity functor as triangle functors. We prove that $A\mbox{-mod}$ is $\mathbf{D}$-standard if and only if any derived equivalence from $\mathbf{D}^b(A\mbox{-mod})$ is standard; see Theorem \ref{thm:2}. Therefore, the open question is equivalent to the following conjecture: any module category $A\mbox{-mod}$ is $\mathbf{D}$-standard.

This notion of $\mathbf{D}$-standardness applies to any $k$-linear abelian category. It is well known that for the categories of coherent sheaves on smooth projective varieties,  derived equivalences of \emph{Fourier-Mukai type} are geometric analogue to standard equivalences. By \cite{Or}, any $k$-linear triangle equivalence between the bounded derived categories of coherent sheaves on smooth projective varieties is of Fourier-Mukai type. Indeed, the proof therein implies that the abelian category of coherent sheaves on any  projective variety is $\mathbf{D}$-standard; compare Proposition \ref{prop:Orlov}. In a certain sense, this fact supports the above conjecture.

Analogously to the above consideration, we introduce  the notion of a \emph{$\mathbf{K}$-standard} additive category, where a pseudo-identity on the bounded homotopy category is involved. We prove that if the category of projective $A$-modules is $\mathbf{K}$-standard, then $A\mbox{-mod}$ is $\mathbf{D}$-standard; see Theorem \ref{thm:1}. This seems to shed new light on the above conjecture.

The paper is structured as follows. In Section 2, we recall basic facts on triangle functors and centers. The notions of a pseudo-identity on the bounded homotopy category of an additive category  and on the bounded derived category of an abelian category are introduced in Section 3.  In Section 4, we introduce the notion of a (strongly) $\mathbf{K}$-standard additive category,  and observe that an Orlov category  \cite{AR} is strongly $\mathbf{K}$-standard; see Proposition \ref{prop:AR}. Analogously, we have  the notion of a (strongly) $\mathbf{D}$-standard abelian category  in Section 5, where we observe that an abelian category with an ample sequence \cite{Or} of objects is strongly $\mathbf{D}$-standard; see Proposition \ref{prop:Orlov}. We prove Theorem \ref{thm:2}, which relates the above open question to the $\mathbf{D}$-standardness.

In Section 6, we prove Theorem \ref{thm:1}, which claims that  an abelian category with enough projectives is $\mathbf{D}$-standard provided that the full subcategory of projectives is $\mathbf{K}$-standard. In the final section, we provide two examples of algebras,  whose module categories are $\mathbf{D}$-standard. In particular, the algebra of dual numbers provides a $\mathbf{D}$-standard, but not strongly $\mathbf{D}$-standard, module category; see Theorem \ref{thm:exam1}.

Throughout the paper, we assume that the categories in consideration are skeletally small. We assume the axiom of global choice, that is, the axiom of choice for classes.

\section{Triangle functors and centers}

In this section, we recall basic facts on triangle functors and the centers of triangulated categories.

\subsection{Triangle functors and extending natural transformations}

Let $\mathcal{T}$ and $\mathcal{T}'$ be triangulated categories, whose translation functors are denoted by $\Sigma$ and $\Sigma'$, respectively. Recall that a triangle functor $(F, \omega)$ consists of an additive functor $F\colon \mathcal{T}\rightarrow \mathcal{T}'$ and a natural isomorphism $\omega\colon F\Sigma\rightarrow \Sigma' F$ such that any exact triangle $X\rightarrow Y\rightarrow Z\stackrel{h}\rightarrow \Sigma(X)$ in $\mathcal{T}$ is sent to an exact triangle $F(X)\rightarrow F(Y) \rightarrow F(Z) \xrightarrow{\omega_X\circ F(h)} \Sigma'F(X)$ in $\mathcal{T}'$.

The natural isomorphism $\omega$ is called the \emph{connecting isomorphism} for $F$. When $\omega$ is understood or not important in the context, we suppress it and write $F$ for the triangle functor $(F, \omega)$. The connecting isomorphism $\omega$ is \emph{trivial} if $F\Sigma=\Sigma'F$ and $\omega={\rm Id}_{F\Sigma}$ is the identity transformation. For example, the identity functor ${\rm Id}_\mathcal{T}$, as a triangle functor, is understood as the pair $({\rm Id}_\mathcal{T}, {\rm Id}_\Sigma)$, which has the trivial connecting isomorphism.

For a triangle functor $(F, \omega)$, we define natural isomorphisms $\omega^n\colon F\Sigma^n\rightarrow {\Sigma'}^nF$ for all $n\geq 1$ as follows: $\omega^1=\omega$ and $\omega^{n+1}={\Sigma'}^n\omega\circ \omega^n\Sigma$ for $n\geq 1$. We observe $\omega^{n+1}=\Sigma'\omega^n\circ \omega \Sigma^n$. If both $\Sigma$ and $\Sigma'$ are automorphisms of categories, we define natural isomorphisms $\omega^{-n}\colon F\Sigma^{-n}\rightarrow {\Sigma'}^{-n}F$ as follows: $\omega^{-1}=({\Sigma'}^{-1}\omega^{1}\Sigma^{-1})^{-1}$ and $\omega^{-n-1}={\Sigma'}^{-n}\omega^{-1} \circ \omega^{-n}\Sigma^{-1}$ for $n\geq 1$. By convention, $\omega^0={\rm Id}_F$.

For two triangle functors $(F, \omega)$ and $(F', \omega')$ from $\mathcal{T}$ to $\mathcal{T}'$, a natural transformation $\eta\colon (F, \omega)\rightarrow (F', \omega')$ between triangle functors means a natural transformation $\eta\colon F\rightarrow F'$ satisfying $\omega'\circ \eta\Sigma=\Sigma' \eta\circ \omega$. The composition of two triangle functors $(F, \omega)\colon \mathcal{T}\rightarrow \mathcal{T}'$ and $(G, \gamma)\colon \mathcal{T}'\rightarrow \mathcal{T}''$ is given by $(GF, \gamma F\circ G\omega)\colon \mathcal{T}\rightarrow \mathcal{T}''$, which is often denoted just by $GF$.

The following fact is well known.

\begin{lem}\label{lem:iso}
Let $\eta\colon (F, \omega)\rightarrow (F', \omega')$ be the natural transformation as above. Then the full subcategory ${\rm Iso}(\eta)=\{X\in \mathcal{T}\; |\; \eta_X \mbox{ is an isomorphism}\}$ of $\mathcal{T}$ is a  triangulated subcategory. \hfill $\square$
\end{lem}

We say that  a full subcategory $\mathcal{S}$ of $\mathcal{T}$ is \emph{generating} provided that the smallest triangulated subcategory containing $\mathcal{S}$ is $\mathcal{T}$ itself. The following well-known result is known as Beilinson's Lemma; see \cite[II.3.4]{Hap88}.

\begin{lem}\label{lem:Bei}
Let $F\colon \mathcal{T}\rightarrow \mathcal{T}'$ be a triangle functor. Assume that $\mathcal{S}\subseteq \mathcal{T}$ is a generating subcategory. Then $F$ is fully faithful if and only if $F$ induces isomorphisms
$${\rm Hom}_\mathcal{T}(X, \Sigma^n(Y))\longrightarrow {\rm Hom}_{\mathcal{T}'}(FX, F\Sigma^n(Y))$$
for all $X, Y\in \mathcal{S}$ and $n\in \mathbb{Z}$. In this case, $F$ is dense if and only if the essential image ${\rm Im}\;F$ contains a generating subcategory of $\mathcal{T}'$. \hfill $\square$
\end{lem}

Let $F\colon \mathcal{C}\rightarrow \mathcal{D}$ be a functor. For each object $C$ in $\mathcal{C}$, we choose an object $F'(C)$ in $\mathcal{D}$ and an isomorphism $\delta_C\colon F(C)\rightarrow F'(C)$. Here, we are using the axiom of choice for the class of objects in $\mathcal{C}$. We call these chosen isomorphisms $\delta_C$'s the \emph{adjusting isomorphisms}.

Indeed, the choice makes $F'$ into a functor such that $\delta$ is a natural  isomorphism between $F$ and $F'$. The action of $F'$ on a morphism $f\colon C\rightarrow C'$ is given by $F'(f)=\delta_{C'} \circ F(f)\circ \delta_{C}^{-1}$. It follows that $F'$ preserves the identity morphism and the composition of morphisms, that is, it is indeed a functor. By the very construction of $F$, we observe the naturality of  $\delta$. In a certain sense, the new functor $F'\colon \mathcal{C}\rightarrow \mathcal{D}$ is \emph{adjusted} from the given functor $F$.

By the following well-known lemma, we might also adjust triangle functors.

\begin{lem}\label{lem:iso-tri}
Let $(F, \omega)\colon \mathcal{T}\rightarrow \mathcal{T}'$ be a triangle functor. Assume that $F'\colon \mathcal{T}\rightarrow \mathcal{T}'$ is another functor with a natural isomorphism  $\delta\colon F\rightarrow F'$. Then there is a unique isomorphism $\omega'\colon F'\Sigma\rightarrow \Sigma' F'$ such that $(F', \omega')$ is a triangle functor and that $\delta$ is an isomorphism between triangle functors.
\end{lem}

\begin{proof}
Take $\omega'=\Sigma' \delta\circ \omega\circ (\delta\Sigma)^{-1}$. The statements are direct to verify.
\end{proof}

The following standard fact will be used later. We provide a full proof for completeness.

\begin{lem}\label{lem:twofun}
Let $F, G\colon \mathcal{A}\rightarrow \mathcal{B}$ be two additive functors between additive categories. Assume that $\mathcal{C}\subseteq \mathcal{A}$ is a full subcategory such that any object in $\mathcal{A}$ is isomorphic to a finite direct sum of objects in $\mathcal{C}$. Let $\eta\colon F|_\mathcal{C}\rightarrow G|_\mathcal{C}$ be a natural transformation. Then there is a unique natural transformation $\eta'\colon F\rightarrow G$ extending $\eta$. Moreover, if $\eta$ is an isomorphism, so is $\eta'$.
\end{lem}

\begin{proof}
For each object $A$ in $\mathcal{A}$, we choose an isomorphism $\xi_A\colon A\rightarrow \oplus_{i\in I}C_i$ with each $C_i\in \mathcal{C}$ and $I$ a finite set. We make the choice such that $\xi_C$ is the identity morphism for each object $C$ in $\mathcal{C}$. Here, we use the axiom of choice for the class of objects in $\mathcal{A}$.

We define $\eta'_A\colon F(A)\rightarrow G(A)$ to be $G(\xi_A)^{-1}\circ (\oplus_{i\in I}\eta_{C_i})\circ F(\xi_A)$. Here, we identify $F(\oplus_{i\in I}C_i)$ with $\oplus_{i\in I}F(C_i)$, $G(\oplus_{i\in I}C_i)$ with $\oplus_{i\in I}G(C_i)$. In particular, $\eta'_C=\eta_C$ for $C\in \mathcal{C}$ by our choice.

We claim that  the morphism  $\eta'_A$ is natural in $A$. For this, we take an arbitrary morphism $f\colon A\rightarrow A'$ in $\mathcal{A}$. For the object $A'$, we have the chosen isomorphism $\xi_{A'}\colon A'\rightarrow \oplus_{j\in J }C'_j$ with $C'_j\in \mathcal{C}$ and $J$ is a finite set. Then we have the following commutative diagram
\[
\xymatrix{
A\ar[d]_-{f}\ar[rr]^-{\xi_A} && \bigoplus_{i\in I} C_i\ar[d]^-{(f_{ji})_{(i, j)\in I\times J} }\\
A'\ar[rr]^-{\xi_{A'}} && \bigoplus_{j\in J} C'_j
}
\]
Here, each $f_{ji}\colon C_i\rightarrow C'_j$ is a morphism in $\mathcal{C}$. In the following commutative diagram, the middle square uses the naturality of $\eta$.
\[
\xymatrix{
F(A) \ar[d]_-{F(f)} \ar[r]^-{F(\xi_A)} & \bigoplus_{i\in I} F(C_i)\ar[d]^-{(F(f_{ji}))}  \ar[rr]^-{\bigoplus_{i\in I}\eta_{C_i}} && \bigoplus_{i\in I}G(C_i)  \ar[d]^-{(G(f_{ji}))} \ar[r]^-{G(\xi_A)^{-1}} & G(A)\ar[d]^-{G(f)}\\
F(A')  \ar[r]^-{F(\xi_{A'})} & \bigoplus_{j\in J} F(C'_j) \ar[rr]^-{\bigoplus_{j\in J}\eta_{C'_j}} && \bigoplus_{j\in J}G(C'_j) \ar[r]^-{G(\xi_{A'})^{-1}} & G(A')
}\]
The outer commutative  diagram proves that $\eta'_{A'}\circ F(f)=G(f)\circ \eta'_A$, as required.

 We mention that the above argument actually proves that $\eta'_A$ is independent of the choice of the isomorphism $\xi_A$, by taking $A'=A$ and $f={\rm Id}_A$.
\end{proof}

Let $k$ be a commutative ring. Let $\mathcal{A}$ be a $k$-linear additive category. For a set $\mathcal{M}$ of morphisms in $\mathcal{A}$, we denote by ${\rm obj}(\mathcal{M})$ the full subcategory formed by those objects, which are either the domain or the codomain of a morphism in $\mathcal{M}$. We say that  $\mathcal{M}$  linearly \emph{spans} $\mathcal{A}$ provided that each morphism in ${\rm obj}(\mathcal{M})$ is a $k$-linear combination of the identity morphisms and  composition of morphisms from $\mathcal{M}$, and that each object in $\mathcal{A}$ is isomorphic to a finite direct sum of objects in ${\rm obj}(\mathcal{M})$.

\begin{lem}\label{lem:iso-gen}
Let $\mathcal{M}$ be a spanning set of morphisms in $\mathcal{A}$. Assume that $F\colon \mathcal{A}\rightarrow \mathcal{A}$ is a $k$-linear endofunctor such that $F(f)=f$ for each $f\in \mathcal{M}$. Then there is a unique natural  isomorphism $\theta\colon F\rightarrow {\rm Id}_\mathcal{A}$ satisfying $\theta_S={\rm Id}_S$ for each object $S$ from ${\rm obj}(\mathcal{M})$.
\end{lem}

\begin{proof}
The assumption implies that $F(S)=S$ for any object $S$ from ${\rm obj}(\mathcal{M})$. Moreover, the restriction of $F$ on ${\rm obj}(\mathcal{M})$ is the identity functor, since it acts on morphisms by the identity. Applying Lemma \ref{lem:twofun} to $F$ and ${\rm Id}_\mathcal{A}$, we are done.
\end{proof}

\subsection{Almost vanishing morphisms and centers}

Throughout this subsection,  $k$ will be a field and $\mathcal{T}$ will be a $k$-linear triangulated category, which is Hom-finite and Krull-Schmidt.

Following \cite[Definition 2.1]{Ku}, a nonzero morphism $w\colon Z\rightarrow X$  in $\mathcal{T}$ is \emph{almost vanishing} provided that $f\circ w=0$ and $w\circ g=0$ for any non-section $f\colon X\rightarrow A$ and non-retraction $g\colon B\rightarrow Z$. This happens if and only if $w$ fits into an almost split triangle $\Sigma^{-1}X\rightarrow E\rightarrow Z\stackrel{w}\rightarrow X$; see \cite[I.4.1]{Hap88}. In particular, both $Z$ and $X$ are indecomposable.

\begin{prop}\label{prop:almvan}
 Assume that $X\stackrel{f}\rightarrow Y \stackrel{g}\rightarrow Z \stackrel{h}\rightarrow \Sigma X$ is an exact triangle in $\mathcal{T}$ with $g\neq 0$ and $h\neq 0$ such that ${\rm End}_\mathcal{T}(Z)$ either equals $k$ or $k{\rm Id}_Z\oplus k\Delta$, where the morphism $\Delta\colon Z\rightarrow Z$ is almost vanishing. Then for a nonzero scalar $\lambda$, the triangle $X\stackrel{f}\rightarrow Y \stackrel{g}\rightarrow Z \xrightarrow{\lambda h} \Sigma(X)$ is exact if and only if $\lambda =1$.
\end{prop}

\begin{proof}
We observe that $g$ is a non-retraction, otherwise $h=0$. Similarly, $h$ is a non-section. Assume that the given triangle is exact. Then we have an isomorphism $\xi\colon Z\rightarrow Z$ making the following diagram commute.
\[\xymatrix{
X \ar[r]^-{f} \ar@{=}[d]& Y \ar@{=}[d] \ar[r]^-{g} & Z \ar@{.>}[d]^-{\xi} \ar[r]^-{h}  & \Sigma(X) \ar@{=}[d]\\
X \ar[r]^-{f} & Y \ar[r]^-{g} & Z \ar[r]^-{\lambda h}  & \Sigma(X)}
\]
If ${\rm End}_\mathcal{T}(Z)=k$, we assume that $\xi=\mu {\rm Id}_Z$ for some $\mu\in k$. It follows from the middle square that $\mu=1$, and thus $\lambda=1$ from the right square.

In the second case, we assume that $\xi=\mu {\rm Id}_Z+\gamma \Delta$ for some $\mu, \gamma\in k$. By the middle square and the fact that $\Delta\circ g=0$, we have $\mu=1$. By the right square and the fact that $h\circ \Delta=0$, we infer that $\lambda=1$.
\end{proof}

We denote by ${\rm ind}\mathcal{T}$ a complete set of representatives of isomorphism classes of indecomposable objects in $\mathcal{T}$. Here, ${\rm ind}\mathcal{T}$ is indeed a set, since $\mathcal{T}$ is skeletally small. Denote by $\Lambda$ the subset consisting of these objects $X$ with an almost vanishing morphism $\Delta_X\colon X\rightarrow X$ such that $\Delta_X$ is central in ${\rm End}_\mathcal{T}(X)$.

 The following is a variant of \cite[Lemma 2.2]{Ku}; compare \cite[Remark 4.15]{Rou2}.

\begin{lem}\label{lem:Ric} For each $X\in \Lambda$, we associate a scalar $\lambda_X$. Then there is a unique natural isomorphism $\eta\colon{\rm Id}_\mathcal{T}\rightarrow {\rm Id}_\mathcal{T}$ such that $\eta_X={\rm Id}_X+\lambda_X\Delta_X$ for $X\in \Lambda$ and $\eta_Y={\rm Id}_Y$ for $Y\in {\rm ind}\mathcal{T}\backslash\Lambda$.
\end{lem}

\begin{proof}
We view ${\rm ind}\mathcal{T}$ as a full subcategory of $\mathcal{T}$. By Lemma \ref{lem:twofun}, it suffices to verify that the restriction  of the isomorphism $\eta$ on ${\rm ind}\mathcal{T}$ is natural. For this, we take an arbitrary  morphism $f\colon Y\rightarrow Z$ with $Y, Z\in {\rm ind}\mathcal{T}$. We claim that $\eta_Z\circ f=f\circ \eta_Y$.

If neither $Y$ nor $Z$ lies in $\Lambda$, the claim is clear. If $Y$ lies in $\Lambda$ and $Z$ does not lie in $\Lambda$, we have $f\circ \Delta_Y=0$, since $f$ is a non-section and $\Delta_Y$ is almost vanishing. Then the claim follows. The same argument works for the case $Y\in \Lambda$ and $Z\notin \Lambda$.

 For the rest, we may assume that both $Y$ and $Z$ lie in $\Lambda$.  If $Y=Z$, then the claims follows, since $\Delta_Y$ and thus $\eta_Y$ are central in ${\rm End}_\mathcal{T}(Y)$. If $Y\neq Z$, we have $\Delta_Z\circ f=0=f\circ \Delta_Y$. This implies the claim. We are done.
\end{proof}

Let $\mathcal{A}$ be a $k$-linear  additive category. We denote by $Z(\mathcal{A})$ the \emph{center} of $\mathcal{A}$, which is by definition the set of natural transformations $\lambda\colon {\rm Id}_\mathcal{A}\rightarrow {\rm Id}_\mathcal{A}$. To ensure that $Z(\mathcal{A})$ is indeed a set, we use the assumption that $\mathcal{A}$ is skeletally small. Then $Z(\mathcal{A})$ is a commutative $k$-algebra, whose addition and multiplication are induced by the addition and composition of natural transformations, respectively.

 We denote by $Z_\vartriangle (\mathcal{T})$ the \emph{triangle center} of $\mathcal{T}$, which is the set of natural transformations $\lambda\colon {\rm Id}_\mathcal{T}\rightarrow {\rm Id}_\mathcal{T}$ between  triangle functors, equivalently, the natural transformation $\lambda$ satisfies $\lambda \Sigma=\Sigma \lambda$. Then $Z_\vartriangle (\mathcal{T})$ is a subalgebra of $Z(\mathcal{\mathcal{T}})$. We mention that $Z_\vartriangle (\mathcal{T})$ is the zeroth component of the graded center of $\mathcal{T}$; compare \cite{Ku, Lin}.

The following observation will be useful.

\begin{lem}\label{lem:nat}
Let $(F, \omega)\colon \mathcal{T}\rightarrow \mathcal{T}$ be a triangle autoequivalence. Then any natural transformation $(F, \omega)\rightarrow (F, \omega)$ of triangle functors is of the form $F\lambda$ for a uniquely determined $\lambda\in Z_\vartriangle(\mathcal{T})$. \hfill $\square$
\end{lem}

Following \cite[Section 4]{CR}, we say that $\mathcal{T}$ is a \emph{block},  provided that $\mathcal{T}$ does not admit a decomposition into the product of  two nonzero triangulated subcategories.  Moreover, it is \emph{non-degenerate} if there is a nonzero non-invertible morphism $X\rightarrow Y$ between some indecomposable objects $X$ and $Y$.

\begin{prop}\label{prop:ZZ}
Let $\mathcal{T}$ be a non-degenerate block such that ${\rm End}_\mathcal{T}(X)=k$ for each indecomposable object $X$. Then the following statements hold.
\begin{enumerate}
\item We have $Z(\mathcal{T})=k=Z_\vartriangle(\mathcal{T})$.
\item If $({\rm Id}_\mathcal{T}, \omega)$ is a triangle functor, then $\omega={\rm Id}_\Sigma$, the identity transformation on $\Sigma$.
\end{enumerate}
\end{prop}

\begin{proof}
For (1), it suffices to show that any natural transformation $\eta\colon {\rm Id}_\mathcal{T}\rightarrow {\rm Id}_\mathcal{T}$ is given by a scalar.  By assumption, $\eta_X=\lambda_X{\rm Id}_X$ for each indecomposable object $X$ and some scalar $\lambda_X$. In view of Lemma \ref{lem:twofun},  it suffices to show that $\lambda_X=\lambda_Y$ for any indecomposables $X$ and $Y$.

We observe that $\lambda_X=\lambda_Y$ provided that there is a nonzero map $X\rightarrow Y$ or $Y\rightarrow X$, using the naturality of $\eta$. Since $\mathcal{T}$ is a non-degenerate block, for any indecomposables $X$ and $Y$,  there is a sequence $X=X_0, X_1, \cdots, X_n=Y$ such that ${\rm Hom}_\mathcal{T}(X_i, X_{i+1})\neq 0$ or ${\rm Hom}_\mathcal{T}(X_{i+1}, X_i)\neq 0$; see \cite[Proposition 4.2 and Remark 4.7]{CR}. From this sequence we infer that $\lambda_X=\lambda_Y$.

For (2), we observe that $\omega=\Sigma(\eta)$ for a unique $\eta\in Z(\mathcal{T})$. By (1) we may assume that $\eta=\lambda\in k$. Take a nonzero non-invertible morphism $g\colon X\rightarrow Y$ between indecomposables and form an exact triangle  $Z\stackrel{f}\rightarrow X\stackrel{g}\rightarrow Y\stackrel{h}\rightarrow \Sigma(Z)$. Since $X$ is indecomposable, we observe that $h\neq 0$. Applying the triangle functor $({\rm Id}_\mathcal{T}, \omega)$ to this triangle, we obtain an exact triangle
$$Z\stackrel{f}\longrightarrow X\stackrel{g}\longrightarrow Y\stackrel{\lambda h}\longrightarrow \Sigma(Z).$$
By Proposition \ref{prop:almvan}, we infer that $\lambda=1$. Then we are done.
\end{proof}

\section{Pseudo-identities and centers}

In this section, we study triangle endofunctors on the bounded homotopy category of an additive category and on the bounded derived category of an abelian category. We introduce the notion of a pseudo-identity endofunctor on them. Their triangle centers are studied.

\subsection{Pseudo-identities on  bounded homotopy categories}

Let $\mathcal{A}$ be an additive category. We denote by $\mathbf{K}^b(\mathcal{A})$ the homotopy category of bounded complexes in $\mathcal{A}$. A bounded complex $X$ is visualized as follows
$$\cdots \longrightarrow X^n \stackrel{d_X^n} \longrightarrow X^{n+1}\stackrel{d_X^{n+1}}\longrightarrow X^{n+2}\longrightarrow \cdots$$
where $X^n\neq 0$ for only finitely many $n$'s and  the differentials satisfy $d_X^{n+1}\circ d_X^n=0$. The translation functor $\Sigma$ on complexes is defined such that $\Sigma(X)^n=X^{n+1}$ and $d_{\Sigma(X)}^n=-d_X^{n+1}$. For a chain map $f\colon X\rightarrow Y$, the translated chain map $\Sigma(f)\colon \Sigma(X)\rightarrow \Sigma(Y)$ is given by $\Sigma(f)^n=f^{n+1}$ for each $n\in \mathbb{Z}$.

An additive functor $G\colon \mathcal{A}\rightarrow \mathcal{B}$ induces a triangle functor $\mathbf{K}^b(G)\colon \mathbf{K}^b(\mathcal{A})\rightarrow \mathbf{K}^b(\mathcal{B})$, which acts componentwise on complexes and whose connecting isomorphism is trivial. Similarly, a natural transformation $\eta\colon G\rightarrow G'$ induces a natural transformation $\mathbf{K}^b(\eta)\colon \mathbf{K}^b(G)\rightarrow \mathbf{K}^b(G')$ between triangle functors.

For an object $A$ in $\mathcal{A}$, we denote by $A$ the corresponding stalk complex concentrated on degree zero. In this way, we view $\mathcal{A}$ as a full subcategory of $\mathbf{K}^b(\mathcal{A})$. For $A\in \mathcal{A}$ and $n\in \mathbb{Z}$, the corresponding stalk complex $\Sigma^n(A)$ is concentrated on degree $-n$.

For a complex $X$ and $n\in \mathbb{Z}$, we consider the \emph{brutal truncation} $\sigma_{\geq -n}X=\cdots\rightarrow 0\rightarrow X^{-n}\stackrel{d_X^{-n}}\rightarrow X^{1-n}\rightarrow \cdots$, which is a subcomplex of $X$. There is a projection $\pi_n\colon \sigma_{\geq -n}X\rightarrow \Sigma^n(X^{-n})$, and thus  an exact triangle in $\mathbf{K}^b(\mathcal{A})$
\begin{align}\label{eqy:tri}
\Sigma^{n-1}(X^{-n})\stackrel{f}\longrightarrow \sigma_{\geq 1-n}X\stackrel{i_n}\longrightarrow \sigma_{\geq -n}X\stackrel{\pi_n}\longrightarrow \Sigma^n(X^{-n}),
\end{align}
where $i_n$ is the inclusion map and $f$ is given by the minus differential $-d^{-n}_X\colon X^{-n}\rightarrow X^{1-n}$. Using these triangles, one observes that $\mathcal{A}$ is a generating subcategory of $\mathbf{K}^b(\mathcal{A})$.

\begin{lem}\label{lem:K-res}
Let $F\colon \mathbf{K}^b(\mathcal{A})\rightarrow \mathbf{K}^b(\mathcal{A})$ be a triangle functor satisfying $F(\mathcal{A})\subseteq \mathcal{A}$.  Then the following statements hold.
 \begin{enumerate}
 \item $F$ is fully faithful if and only if so is the restriction $F|_\mathcal{A}\colon \mathcal{A}\rightarrow \mathcal{A}$.
     \item If the restriction $F|_\mathcal{A}\colon \mathcal{A}\rightarrow \mathcal{A}$ is an equivalence, so is $F$.
         \item Assume that $\mathcal{A}$ has split idempotents. If $F$  is an equivalence, so is $F|_\mathcal{A}$.
     \end{enumerate}
\end{lem}

\begin{proof}
The ``only if" part of (1) is trivial. For the ``if" part, we observe that ${\rm Hom}_{\mathbf{K}^b(\mathcal{A})}(X, \Sigma^n(Y))=0$ for $X, Y\in \mathcal{A}$ and $n\neq 0$. Since $\mathcal{A}$ is a generating subcategory of $\mathbf{K}^b(\mathcal{A})$, we apply Lemma \ref{lem:Bei} to obtain that $F$ is fully faithful.

For (2), we observe that if $F|_\mathcal{A}$ is an equivalence, the essential image ${\rm Im}\; F$ contains $\mathcal{A}$, a generating subcategory of $\mathbf{K}^b(\mathcal{A})$. In view of the second statement of Lemma \ref{lem:Bei}, we infer that $F$ is dense.

For (3), we recall the following well-known observation: a bounded complex $Y$ is isomorphic to some object in $\mathcal{A}$ if and only if ${\rm Hom}_{\mathbf{K}^b(\mathcal{A})}(Y, \Sigma^n(A))=0={\rm Hom}_{\mathbf{K}^b(\mathcal{A})}(\Sigma^n(A), Y)$ for each $A\in \mathcal{A}$ and $n\neq 0$.

It suffices to prove that for any complex $X$, if $F(X)$ is isomorphic to some object in $\mathcal{A}$, so is $X$. For each $A\in\mathcal{A}$ and $n\neq 0$, we have
\begin{align*}
{\rm Hom}_{\mathbf{K}^b(\mathcal{A})}(X, \Sigma^n(A))&\simeq {\rm Hom}_{\mathbf{K}^b(\mathcal{A})}(F(X), F\Sigma^n(A))\\
&\simeq {\rm Hom}_{\mathbf{K}^b(\mathcal{A})}(F(X), \Sigma^n(FA))=0,
 \end{align*}
where the first isomorphism uses the fully-faithfulness of $F$ and the last equality uses the fact that $F(A)\in \mathcal{A}$. Similarly, we have ${\rm Hom}_{\mathbf{K}^b(\mathcal{A})}(\Sigma^n(A), X)=0$.  Then we are done by the above observation.
\end{proof}

The following result is analogous to \cite[Proposition 7.1]{Ric89}, where a completely different argument is used.

\begin{prop}\label{prop:K-pseudo}
Let $(F, \omega)\colon \mathbf{K}^b(\mathcal{A})\rightarrow \mathbf{K}^b(\mathcal{A})$ be a triangle autoequivalence satisfying $F(\mathcal{A})\subseteq \mathcal{A}$. Assume that the restriction $F|_\mathcal{A}$ is isomorphic to the identity functor ${\rm Id}_\mathcal{A}$.  Then for each complex $X\in \mathbf{K}^b(\mathcal{A})$, $F(X)$ is isomorphic to $X$.
\end{prop}

\begin{proof}
Assume that  $\phi\colon F|_\mathcal{A}\rightarrow {\rm Id}_\mathcal{A}$ is the given isomorphism. Using the translation functor and the connecting isomorphism $\omega$, it suffices to prove the statement under the assumption that $X^i=0$ for $i>0$.

We claim that for each $n\geq 0$,  there is an isomorphism
$$a_n\colon F(\sigma_{\geq -n}X)\longrightarrow \sigma_{\geq -n}X$$
 satisfying $\pi_n\circ a_n=\Sigma^n(\phi_{X^{-n}})\circ \omega^n_{X^{-n}}\circ F(\pi_n)$. The claim will be proved by induction on $n$.

We take $a_0=\phi_{X^0}$. We assume that the isomorphism $a_{n-1}$ is already given for some $n\geq 1$. Consider the exact triangle (\ref{eqy:tri}).  We claim that the left square in the following diagram commutes.
\[\xymatrix{
F\Sigma^{n-1}(X^{-n}) \ar[r]^-{F(f)} \ar[d]^-{\Sigma^{n-1}(\phi_{X^{-n}})\circ \omega^{n-1}_{X^{-n}}} & F(\sigma_{\geq 1-n}X) \ar[d]^-{a_{n-1}}\ar[r]^-{F(i_n)} & F(\sigma_{\geq -n}X) \ar[rr]^-{\omega_{\Sigma^{n-1}(X^{-n})}\circ F(\pi_n)} && \Sigma F\Sigma^{n-1}(X^{-n})\ar[d]_{\Sigma^n(\phi_{X^{-n}})\circ \Sigma(\omega^{n-1}_{X^{-n}})}\\
\Sigma^{n-1}(X^{-n}) \ar[r]^-{f} & \sigma_{\geq 1-n}X \ar[r]^-{i_n} & \sigma_{\geq -n}X\ar[rr]^-{\pi_n} && \Sigma^n(X^{-n})
}\]
Indeed, the following map induced by $\pi_{n-1}\colon \sigma_{\geq 1-n}X\rightarrow \Sigma^{n-1}(X^{1-n})$ is injective
$${\rm Hom}_{\mathbf{K}^b(\mathcal{A})} (F\Sigma^{n-1}(X^{-n}), \sigma_{\geq 1-n}X)\longrightarrow {\rm Hom}_{\mathbf{K}^b(\mathcal{A})} (F\Sigma^{n-1}(X^{-n}), \Sigma^{n-1}(X^{1-n})).$$
Hence, for the claim, it suffices to prove
 $$\pi_{n-1}\circ a_{n-1}\circ F(f)= \pi_{n-1}\circ  f\circ  \Sigma^{n-1}(\phi_{X^{-n}})\circ \omega^{n-1}_{X^{-n}}.$$
By the induction hypothesis,  the first equality  in the following identity holds:
\begin{align*}
\pi_{n-1}\circ a_{n-1}\circ F(f) & = \Sigma^{n-1}(\phi_{X^{1-n}})\circ \omega^{n-1}_{X^{1-n}}\circ F(\pi_{n-1})\circ F(f)\\
&=-\Sigma^{n-1}(\phi_{X^{1-n}})\circ \omega^{n-1}_{X^{1-n}}\circ F\Sigma^{n-1}(d^{-n}_X)\\
&=-\Sigma^{n-1}(d^{-n}_X)\circ \Sigma^{n-1}(\phi_{X^{-n}})\circ \omega^{n-1}_{X^{-n}}\\
& =\pi_{n-1}\circ f\circ \Sigma^{n-1}(\phi_{X^{-n}})\circ \omega^{n-1}_{X^{-n}}.
\end{align*}
Here, the second and fourth equalities use the fact that $\pi_{n-1}\circ f=-\Sigma^{n-1}(d^{-n}_X)$, and the third uses the naturality of $\omega^{n-1}$ and $\phi$.

Thanks to the above diagram between exact triangles, the required isomorphism $a_n\colon F(\sigma_{\geq -n}X)\rightarrow \sigma_{\geq -n}X$ follows from the axiom (TR3) in the triangulated structure of $\mathbf{K}^b(\mathcal{A})$.
\end{proof}

Inspired by the above result, it seems to be of interest to have the following notion. For each $n\in \mathbb{Z}$, we denote by $\Sigma^n(\mathcal{A})$ the full subcategory of $\mathbf{K}^b(\mathcal{A})$ consisting of stalk complexes concentrated on degree $-n$. We identify  $\Sigma^0(\mathcal{A})$ with $\mathcal{A}$.

\begin{defn}
 A triangle functor $(F, \omega)\colon \mathbf{K}^b(\mathcal{A})\rightarrow \mathbf{K}^b(\mathcal{A})$ is called a \emph{pseudo-identity},  provided that $F(X)=X$ for each bounded complex $X$ and that its restriction $F|_{\Sigma^n(\mathcal{A})}$ to the subcategory $\Sigma^n(\mathcal{A})$ equals the identity functor on $\Sigma^n(\mathcal{A})$,  for each $n\in \mathbb{Z}$. \hfill $\square$
\end{defn}

The difference between a pseudo-identity and the genuine identity functor on $\mathbf{K}^b(\mathcal{A})$ lies in their action on morphisms and their connecting isomorphisms.

\begin{cor}\label{cor:K-pseudo}
Let $(F, \omega)\colon \mathbf{K}^b(\mathcal{A})\rightarrow \mathbf{K}^b(\mathcal{A})$  be a triangle functor. Then $(F, \omega)$ is isomorphic to a pseudo-identity if and only if $F$ is an autoequivalence satisfying  $F(\mathcal{A})\subseteq \mathcal{A}$  such that the restriction $F|_\mathcal{A}$ is isomorphic to the identity functor.
\end{cor}

\begin{proof}
By Lemma \ref{lem:K-res}, a pseudo-identity is an autoequivalence. Then we have the ``only if" part.

For the ``if" part, we assume that $\gamma \colon F|_\mathcal{A}\rightarrow {\rm Id}_\mathcal{A}$ is the given isomorphism. We apply Proposition \ref{prop:K-pseudo} and choose for each complex $X$ an isomorphism $\delta_X\colon F(X)\rightarrow X=F'(X)$ such that for each object $A\in \mathcal{A}$ and $n\in \mathbb{Z}$, $\delta_{\Sigma^n(A)}\colon F(\Sigma^nA)\rightarrow F'(\Sigma^nA)$ equals $\Sigma^n(\gamma_A)\circ \omega^n_A$; here, we refer to Subsection 2.1 for the notation $\omega^n$. Using $\delta_X$'s as the adjusting isomorphisms and Lemma \ref{lem:iso-tri}, we obtain a pseudo-identity $(F', \omega')$ on $\mathbf{K}^b(\mathcal{A})$, which is isomorphic to $(F, \omega)$ as triangle functors.
\end{proof}

\begin{lem}\label{lem:K-pseudo}
Let $(F, \omega)\colon \mathbf{K}^b(\mathcal{A})\rightarrow \mathbf{K}^b(\mathcal{A})$   be a  pseudo-identity. Assume that $(F, \omega)$ is isomorphic to the identity functor ${\rm Id}_{\mathbf{K}^b(\mathcal{A})}$, as triangle functors. Then there is a natural isomorphism $\theta\colon (F, \omega)\rightarrow {\rm Id}_{\mathbf{K}^b(\mathcal{A})}$ of triangle functors, whose restriction to $\mathcal{A}$ is the identity transformation.
\end{lem}

\begin{proof}
Take a natural isomorphism $\delta\colon (F, \omega)\rightarrow {\rm Id}_{\mathbf{K}^b(\mathcal{A})}$. The restriction of $\delta$ to $\mathcal{A}$ is an invertible element $\mu$ in $Z(\mathcal{A})$. Set $\theta=\mathbf{K}^b(\mu^{-1})\circ \delta$. Then we are done.
\end{proof}

\subsection{Pseudo-identities on bounded derived categories}

Throughout this subsection, $\mathcal{A}$ is an abelian category. We denote by $\mathbf{D}^b(\mathcal{A})$ the bounded derived category. We identify $\mathcal{A}$ as the full subcategory of $\mathbf{D}^b(\mathcal{A})$ formed by stalk complexes concentrated on degree zero.

An exact functor $G\colon \mathcal{A}\rightarrow \mathcal{B}$ between abelian categories induces a triangle functor $\mathbf{D}^b(G)\colon \mathbf{D}^b(\mathcal{A})\rightarrow \mathbf{D}^b(\mathcal{B})$, which acts componentwise on complexes and has a trivial connecting isomorphism. Similarly, a natural transformation $\mu\colon G\rightarrow G'$ between exact functors induces a natural transformation $\mathbf{D}^b(\mu)\colon \mathbf{D}^b(G)\rightarrow \mathbf{D}^b(G')$ between triangle functors.

For a bounded complex $X$ and $n\in \mathbb{Z}$, we denote by $H^n(X)$ the $n$-th cohomology.  We recall the \emph{good truncations} $\tau_{\leq n}(X)=\cdots \rightarrow  X^{n-2}\stackrel{d_X^{n-2}}\rightarrow X^{n-1} \rightarrow {\rm Ker}d_X^{n}\rightarrow 0\rightarrow \cdots$ and $\tau_{\geq n} (X)=\cdots \rightarrow 0 \rightarrow {\rm Cok}d_X^{n-1}\rightarrow X^{n+1}\stackrel{d_X^{n+1}}\rightarrow X^{n+2}\rightarrow \cdots$. This gives rise to the truncation functors $\tau_{\leq n}$ and $\tau_{\geq n}$ on $\mathbf{D}^b(\mathcal{A})$. There is a functorial isomorphism $H^n(X)\simeq \Sigma^n\tau_{\geq n}\tau_{\leq n}(X)$.

\begin{lem}\label{lem:D-res}
Let $F\colon \mathbf{D}^b(\mathcal{A})\rightarrow \mathbf{D}^b(\mathcal{A})$ be a triangle functor satisfying $F(\mathcal{A})\subseteq \mathcal{A}$.  Then the following statements hold.
 \begin{enumerate}
 \item $F$ is fully-faithful if and only if so is the restriction $F|_\mathcal{A}\colon \mathcal{A}\rightarrow \mathcal{A}$.
     \item $F$ is an equivalence if and only if so is the restriction $F|_\mathcal{A}$.
     \end{enumerate}
\end{lem}

\begin{proof}
The proof is similar to the one of Lemma \ref{lem:K-res}, since $\mathcal{A}$ is also a generating subcategory of  $\mathbf{D}^b(\mathcal{A})$.  We only prove the ``only if" part of (2), that is, the denseness of $F|_\mathcal{A}$. It suffices to claim that if $F(X)$ is isomorphic to some object in $\mathcal{A}$, so is $X$.

We observe that a complex $X$ is isomorphic to some object in $\mathcal{A}$ if and only if $H^n(X)=0$ for $n\neq 0$. By the assumption that $F(\mathcal{A})\subseteq \mathcal{A}$, we infer that $F$ commutes with the truncation functors $\tau_{\leq n}$ and $\tau_{\geq n}$. Consequently, it commutes with taking cohomologies. More precisely, for each bounded complex $X$  and each $n\in \mathbb{Z}$, there is a natural isomorphism
\begin{align*}
 F|_\mathcal{A}(H^n(X))\stackrel{\sim}\longrightarrow H^n(F(X)).
\end{align*}
Since $F|_\mathcal{A}$ is fully faithful, the claim follows immediately.
\end{proof}

We have the following analogue of Proposition \ref{prop:K-pseudo}; compare \cite[Proposition 7.1]{Ric89}.

\begin{prop}\label{prop:D-pseudo}
Let $F\colon \mathbf{D}^b(\mathcal{A})\rightarrow \mathbf{D}^b(\mathcal{A})$ be a triangle autoequivalence satisfying $F(\mathcal{A})\subseteq \mathcal{A}$. Assume that the restriction $F|_\mathcal{A}$ is isomorphic to the identity functor ${\rm Id}_\mathcal{A}$.  Then for each complex $X\in \mathbf{D}^b(\mathcal{A})$, $F(X)$ is isomorphic to $X$.
\end{prop}

\begin{proof}
The same argument of Proposition \ref{prop:K-pseudo} works, where we still use brutal truncations. It suffices to observe that the projection  $\pi_{n-1}\colon \sigma_{\geq 1-n}X\rightarrow \Sigma^{n-1}(X^{1-n})$ induces an injective map
$${\rm Hom}_{\mathbf{D}^b(\mathcal{A})} (F\Sigma^{n-1}(X^{-n}), \sigma_{\geq 1-n}X)\longrightarrow {\rm Hom}_{\mathbf{D}^b(\mathcal{A})} (F\Sigma^{n-1}(X^{-n}), \Sigma^{n-1}(X^{1-n})),$$
since we have
$${\rm Hom}_{\mathbf{D}^b(\mathcal{A})} (F\Sigma^{n-1}(X^{-n}), \sigma_{\geq 2-n}X)\simeq {\rm Hom}_{\mathbf{D}^b(\mathcal{A})} (\Sigma^{n-1}(X^{-n}), \sigma_{\geq 2-n}X)=0.$$
We omit the details.
\end{proof}

The following definition and corollary are analogous to the ones for the homotopy category. Recall that for each $n\in \mathbb{Z}$,  $\Sigma^n(\mathcal{A})$ denotes the full subcategory of $\mathbf{D}^b(\mathcal{A})$ consisting of stalk complexes concentrated on degree $-n$. As usual, we identify $\Sigma^0(\mathcal{A})$ with $\mathcal{A}$.

 \begin{defn}
A triangle functor $(F, \omega)\colon \mathbf{D}^b(\mathcal{A})\rightarrow \mathbf{D}^b(\mathcal{A})$ is a \emph{pseudo-identity} provided that $F(X)=X$ for each bounded complex $X$ and that its restriction $F|_{\Sigma^n(\mathcal{A})}$ to $\Sigma^n(\mathcal{A})$ equals the identity functor on $\Sigma^n(\mathcal{A})$ for each $n\in \mathbb{Z}$. \hfill $\square$
\end{defn}

The following result is a consequence of Lemma \ref{lem:D-res} and Proposition \ref{prop:D-pseudo}.

\begin{cor}\label{cor:D-pseudo}
Let $(F, \omega)\colon \mathbf{D}^b(\mathcal{A})\rightarrow \mathbf{D}^b(\mathcal{A})$  be a triangle functor. Then $(F, \omega)$ is isomorphic to a pseudo-identity if and only if $F$ is an autoequivalence satisfying  $F(\mathcal{A})\subseteq \mathcal{A}$  such that the restriction $F|_\mathcal{A}$ is isomorphic to the identity functor. \hfill $\square$
\end{cor}

The following is analogous to Lemma \ref{lem:K-pseudo}.

\begin{lem}\label{lem:D-pseudo}
Let $(F, \omega)\colon \mathbf{D}^b(\mathcal{A})\rightarrow \mathbf{D}^b(\mathcal{A})$   be a  pseudo-identity. Assume that $(F, \omega)$ is isomorphic to the identity functor ${\rm Id}_{\mathbf{D}^b(\mathcal{A})}$, as triangle functors. Then there is a natural isomorphism $\theta\colon (F, \omega)\rightarrow {\rm Id}_{\mathbf{D}^b(\mathcal{A})}$ of triangle functors, whose restriction to $\mathcal{A}$ is the identity transformation. \hfill $\square$
\end{lem}

\subsection{Comparing centers}

We will compare the triangle centers of the homotopy category and the derived category.

Let $\mathcal{P}$ be an additive category. There is a ring homomorphism
\begin{align}\label{equ:CK}
{\rm res} \colon Z_\vartriangle(\mathbf{K}^b(\mathcal{P}))\longrightarrow Z(\mathcal{P}), \; \lambda\mapsto \lambda|_\mathcal{P}
\end{align}
sending $\lambda$ to its restriction on $\mathcal{P}$. It is surjective. Indeed, there is another canonical ring homomorphism
\begin{align*}
{\rm ind}\colon Z(\mathcal{P})  \longrightarrow Z_\vartriangle(\mathbf{K}^b(\mathcal{P})), \; \mu\mapsto \mathbf{K}^b(\mu),
\end{align*}
which sends $\mu\colon {\rm Id}_\mathcal{P}\rightarrow {\rm Id}_\mathcal{P}$ to $\mathbf{K}^b(\mu)\colon {\rm Id}_{\mathbf{K}^b(\mathcal{P})}\rightarrow {\rm Id}_{\mathbf{K}^b(\mathcal{P})}$. More precisely, the action of $\mathbf{K}^b(\mu)$ on complexes  is componentwise by $\mu$. Since the composition ${\rm res}\circ {\rm ind}$ equals the identity, the homomorphism (\ref{equ:CK}) is surjective.

The following notation is needed. For a class  $S$ of objects in a triangulated category $\mathcal{T}$, we denote by $\langle S\rangle $ the smallest full additive subcategory containing $S$ and closed under taking direct summands, $\Sigma$ and $\Sigma^{-1}$. For two classes $\mathcal{X}$ and $\mathcal{Y}$ of objects, we denote by $\mathcal{X}\star \mathcal{Y}$ the class formed by those objects $Z$, which fit into an exact triangle $X\rightarrow Z\rightarrow Y\rightarrow \Sigma(X)$ for some $X\in \mathcal{X}$ and $Y\in \mathcal{Y}$. We set $\langle S\rangle _1=\langle S\rangle $ and $\langle S\rangle _{d+1}=\langle \langle S\rangle_d \star \langle S\rangle_1\rangle$ for $d\geq 1$.

The following lemma is implicit in \cite[Lemma 4.11]{Rou1}.

\begin{lem}\label{lem:Van}
Let $A\stackrel{a}\rightarrow B\stackrel{b}\rightarrow C$ be two morphisms in $\mathcal{T}$ such that ${\rm Hom}_\mathcal{T}(a, -)$ vanishes on $\mathcal{X}$ and ${\rm Hom}_\mathcal{T}(b, -)$ vanishes on $\mathcal{Y}$. Then ${\rm Hom}_\mathcal{T}(b\circ a, -)$ vanishes on $\mathcal{X} \star \mathcal{Y}$.
\end{lem}

\begin{proof}
Assume that $X\stackrel{u}\rightarrow Z\stackrel{v}\rightarrow Y\rightarrow \Sigma(X)$ is an exact triangle with $X\in \mathcal{X}$ and $Y\in \mathcal{Y}$. Take any morphism $f\colon C\rightarrow Z$. Then $v\circ f\circ b=0$. It follows that $f\circ b=u\circ g$ for some morphism $g\colon B\rightarrow X$. Using $g\circ a=0$, we infer that $f\circ b\circ a=0 $.
\end{proof}

The second statement of the following result is analogous to \cite[Proposition 2.9]{Lin}; compare \cite[Remark 4.12]{Rou1}.

\begin{prop}\label{prop:NK}
Keep the notation as above. Then the kernel $\mathcal{N}$ of the map {\rm res} in  (\ref{equ:CK}) lies in the Jacobson radical of $Z_\vartriangle(\mathbf{K}^b(\mathcal{P}))$.

If $\mathbf{K}^b(\mathcal{P})=\langle \mathcal{P}\rangle _{d}$ for some $d\geq 1$, we have $\mathcal{N}^d=0$.
\end{prop}

\begin{proof}
Let $\lambda\in \mathcal{N}$. Then ${\rm res}(1+\lambda)=1$. In the notation of Lemma \ref{lem:iso}, the triangulated subcategory  ${\rm Iso}(1+\lambda)$ contains $\mathcal{P}$. It forces that ${\rm Iso}(1+\lambda)=\mathbf{K}^b(\mathcal{P})$, that is, $1+\lambda$ is invertible. Consequently,  the ideal $\mathcal{N}$ lies in the Jacobson radical.

For the second statement, we take $\lambda_i\in \mathcal{N}$ for $1\leq i\leq d$. It suffices to claim that for each complex $X$, the composition
$$X\xrightarrow{(\lambda_1)_X} X \longrightarrow \cdots \longrightarrow X\xrightarrow{(\lambda_d)_X} X$$
 is zero. This sequence of morphisms induces a sequence of natural transformations between the Hom  functors on $\mathbf{K}^b(\mathcal{P})$
 $${\rm Hom}(X, -)\xrightarrow{{\rm Hom}((\lambda_1)_X, -)} {\rm Hom}(X, -) \rightarrow \cdots \rightarrow {\rm Hom}(X, -)\xrightarrow{{\rm Hom}((\lambda_d)_X, -)} {\rm Hom}(X, -).$$
We observe that each of these natural transformations vanishes on $\mathcal{P}$. Indeed, for an object $A$ in $\mathcal{P}$ and any morphism $f\colon X\rightarrow A$, $f\circ (\lambda_i)_X=(\lambda_i)_A\circ f=0$.
By Lemma~\ref{lem:Van} the composition vanishes on $\langle \mathcal{P}\rangle_d$, which is equal to $\mathbf{K}^b(\mathcal{P})$. An application of Yoneda's Lemma yields the required claim.
\end{proof}

Let $\mathcal{A}$ be an abelian category. Then there is a ring homomorphism
\begin{align}\label{equ:CD}
{\rm res}\colon Z_\vartriangle(\mathbf{D}^b(\mathcal{A}))\longrightarrow Z(\mathcal{A}), \; \lambda\mapsto \lambda|_\mathcal{A}
\end{align}
sending $\lambda$ to its restriction on $\mathcal{A}$. By a similar argument as above, there is another canonical ring homomorphism
\begin{align*}
{\rm ind}\colon Z(\mathcal{A})  \longrightarrow Z_\vartriangle(\mathbf{D}^b(\mathcal{A})), \; \mu\mapsto \mathbf{D}^b(\mu),
\end{align*}
satisfying that ${\rm res}\circ {\rm ind}$ is equal to the identity. Then the homomorphism (\ref{equ:CD}) is surjective.

The following result is proved by the same argument as Proposition \ref{prop:NK}.

\begin{prop}
Let $\mathcal{A}$ be an abelian category.  Then the kernel $\mathcal{M}$ of the map {\rm res} in (\ref{equ:CD}) lies in the Jacobson radical of $Z_\vartriangle(\mathbf{D}^b(\mathcal{A}))$.

If $\mathbf{D}^b(\mathcal{A})=\langle \mathcal{A}\rangle _{d}$ for some $d\geq 1$, we have $\mathcal{M}^d=0$. \hfill $\square$
\end{prop}

Let $\mathcal{A}$ be an abelian category with enough projective objects. Denote by $\mathcal{P}$ the full subcategory formed by projective objects. We view $\mathbf{K}^b(\mathcal{P})$ as a triangulated  subcategory of $\mathbf{D}^b(\mathcal{A})$.

We consider the following commutative diagram of ring homomorphisms, where ``${\rm res}$" denotes the corresponding restriction of natural transformations
\begin{align}\label{equ:center}
\xymatrix{
Z_\vartriangle(\mathbf{D}^b(\mathcal{A}))  \ar[rr]^-{\stackrel{\rm res}{\sim}} \ar[d]_-{\rm res} && Z_\vartriangle(\mathbf{K}^b(\mathcal{P}))\ar[d]^-{\rm res} \\
Z(\mathcal{A})\ar[rr]^-{\stackrel{\rm res}{\sim}}  && Z(\mathcal{P}).
}
\end{align}
It is well known that the lower row map is an isomorphism. By \cite[Theorem 2.5]{KY} the upper one is also an isomorphism. Consequently, we may identify the kernels of the two vertical homomorphisms.

\section{$\mathbf{K}$-standard additive categories}

In this section, we introduce the notions of a $\mathbf{K}$-standard additive category and a strongly $\mathbf{K}$-standard additive category.

Let $k$ be a commutative ring. We will  assume that all functors and categories are $k$-linear. Throughout, $\mathcal{A}$ is a $k$-linear additive category, which is always assumed to be skeletally small.

\begin{defn}\label{defn:1}
 The category $\mathcal{A}$ is  said to be \emph{$\mathbf{K}$-standard} (over $k$),  provided that the following holds: given any $k$-linear triangle autoequivalence $(F, \omega)\colon \mathbf{K}^b(\mathcal{A})\rightarrow \mathbf{K}^b(\mathcal{A})$  satisfying $F(\mathcal{A})\subseteq \mathcal{A}$ and any natural isomorphism $\theta_0\colon F|_\mathcal{A}\rightarrow {\rm Id}_\mathcal{A}$, there is a natural transformation $\theta\colon (F, \omega)\rightarrow {\rm Id}_{\mathbf{K}^b(\mathcal{A})}$ of triangle functors extending $\theta_0$.

The category $\mathcal{A}$ is said to be \emph{strongly $\mathbf{K}$-standard} (over $k$), if furthermore the above extension $\theta$ is always unique. \hfill $\square$
\end{defn}

We observe that the above extension $\theta$ is necessarily an isomorphism. Indeed, in the notation of Lemma \ref{lem:iso}, the triangulated subcategory ${\rm Iso}(\theta)$ contains $\mathcal{A}$. Then we have ${\rm Iso}(\theta)=\mathbf{K}^b(\mathcal{A})$.

\begin{lem}\label{lem:K-stan}
Let $\mathcal{A}$ be as above. Then the following statements are equivalent.
\begin{enumerate}
\item  The category $\mathcal{A}$ is $\mathbf{K}$-standard.
 \item For any $k$-linear pseudo-identity $(F, \omega)$ on $\mathbf{K}^b(\mathcal{A})$, there is a natural isomorphism $\eta\colon (F, \omega)\rightarrow {\rm Id}_{\mathbf{K}^b(\mathcal{A})}$ of triangle functors such that $\eta|_\mathcal{A}$ is the identity.
     \item Any $k$-linear pseudo-identity $(F, \omega)$ on $\mathbf{K}^b(\mathcal{A})$ is isomorphic to ${\rm Id}_{\mathbf{K}^b(\mathcal{A})}$, as triangle functors.
         \end{enumerate}
\end{lem}

\begin{proof}
The implications $(1)\Rightarrow (2)$ and  $(2)\Rightarrow (3)$ are clear. By Lemma \ref{lem:K-pseudo}, we have $(3)\Rightarrow (2)$.

  For $(2)\Rightarrow (1)$, let $(F, \omega)$ and $\theta_0$ be as in Definition \ref{defn:1}. By Corollary \ref{cor:K-pseudo} and its proof, there is a pseudo-identity $(F', \omega')$ on $\mathbf{K}^b(\mathcal{A})$ with a natural isomorphism $\theta'\colon (F, \omega)\rightarrow (F', \omega')$ such that $\theta'|_\mathcal{A}=\theta_0$. By assumption, there is an isomorphism $\eta\colon (F', \omega')\rightarrow {\rm Id}_{\mathbf{K}^b(\mathcal{A})}$ with $\eta|_\mathcal{A}$ the identity transformation. Take $\theta=\eta\circ \theta'$. Then we are done.
\end{proof}

The centers of the homotopy category and the underlying additive category  play a role for strongly $\mathbf{K}$-standard categories.

\begin{lem}\label{lem:stan-center}
Let $\mathcal{A}$ be a $k$-linear additive category. Then it is strongly $\mathbf{K}$-standard if and only if it is $\mathbf{K}$-standard and the homomorphism {\rm res} in (\ref{equ:CK}) for $\mathcal{A}$ is an isomorphism.
\end{lem}

\begin{proof}
For the ``only if" part, it suffices to show that the homomorphism (\ref{equ:CK}) is injective, since we observe that in Section 3 it is always surjective. We claim  that each $\lambda$ in the kernel of (\ref{equ:CK}) is zero. Indeed, both $1+\lambda$ and $1$ are extensions of  the identity transformation $({\rm Id}_{\mathbf{K}^b(\mathcal{A})})|_\mathcal{A}={\rm Id}_\mathcal{A}\rightarrow {\rm Id}_\mathcal{A}$. By the uniqueness of the extensions,  we infer that $1+\lambda=1$.

For the ``if" part, we take two extensions $\theta, \theta'\colon (F, \omega)\rightarrow {\rm Id}_{\mathbf{K}^b(\mathcal{A})}$ of the given isomorphism $\theta_0\colon F|_\mathcal{A}\rightarrow {\rm Id}_\mathcal{A}$. As mentioned above, both $\theta$ and $\theta'$ are isomorphisms. Then $\theta\circ \theta'^{-1}$ lies in $Z_\vartriangle(\mathbf{K}^b(\mathcal{A}))$, whose restriction to $\mathcal{A}$ is the identity transformation. Since the homomorphism (\ref{equ:CK}) is injective, we infer that $\theta\circ \theta'^{-1}$ is equal to the identity and thus $\theta=\theta'$.
\end{proof}

We have the following basic properties of a $\mathbf{K}$-standard additive category.

\begin{lem}\label{lem:Kdef}
Let $\mathcal{A}$ be a $\mathbf{K}$-standard additive category. Then the following statements hold.
\begin{enumerate}
\item Let $(F, \omega)\colon \mathbf{K}^b(\mathcal{A})\rightarrow \mathbf{K}^b(\mathcal{A})$ be a triangle autoequivalence with $F(\mathcal{A})\subseteq \mathcal{A}$. If $\mathcal{A}$ has split idempotents, then there is an isomorphism $(F, \omega)\stackrel{\sim}\longrightarrow \mathbf{K}^b(F|_\mathcal{A})$ of triangle functors.
    \item Assume further that $\mathcal{A}$ is strongly $\mathbf{K}$-standard. Let $F_1, F_2\colon \mathcal{A}\rightarrow \mathcal{A}$ be two autoequivalences, which are isomorphic. Then any natural transformation $\mathbf{K}^b(F_1)\rightarrow \mathbf{K}^b(F_2)$ of triangle functors is of the form $\mathbf{K}^b(\eta)$ for a unique natural transformation $\eta\colon F_1\rightarrow F_2$.
    \end{enumerate}
    \end{lem}

\begin{proof}
(1) We have observed in Lemma \ref{lem:K-res}(3) that $F|_\mathcal{A}\colon \mathcal{A}\rightarrow \mathcal{A}$ is an autoequivalence. We fix its quasi-inverse $G$. Consider the triangle autoequivalence $\mathbf{K}^b(G)F$, whose restriction to $\mathcal{A}$ is isomorphic to the identity functor. By the $\mathbf{K}$-standard property, we infer that $\mathbf{K}^b(G)F$ is isomorphic to the identity functor. Consequently, we have that $F$ is isomorphic to $\mathbf{K}^b(F|_\mathcal{A})$.

(2) We fix a natural isomorphism $\delta\colon F_2\rightarrow F_1$. Take any natural transformation $\theta\colon \mathbf{K}^b(F_1)\rightarrow \mathbf{K}^b(F_2)$ of triangle functors and set $\eta=\theta|_\mathcal{A}$ to be its restriction to $\mathcal{A}$. By Lemma \ref{lem:nat} there are $\gamma, \gamma'\in Z_\vartriangle(\mathbf{K}^b(\mathcal{A}))$ satisfying $\mathbf{K}^b(F_1)\gamma=\mathbf{K}^b(\delta)\circ \theta$ and $\mathbf{K}^b(F_1)\gamma'=\mathbf{K}^b(\delta)\circ \mathbf{K}^b(\eta)$. We observe that the restrictions of $\gamma$ and $\gamma'$ to $\mathcal{A}$ coincide. Lemma \ref{lem:stan-center} implies that the homomorphism (\ref{equ:CK}) is injective. It follows that $\gamma=\gamma'$, which proves that $\theta=\mathbf{K}^b(\eta)$.
\end{proof}

An additive category $\mathcal{A}$ is \emph{split} provided that it has split idempotents and every morphism $f\colon X\rightarrow Y$ admits a factorization $f=v\circ u$ with $u$ a retraction and $v$ a section.

The following observation provides a trivial example for strongly $\mathbf{K}$-standard categories.

\begin{lem}
Let $\mathcal{A}$ be a split category. Then $\mathcal{A}$ is strongly $\mathbf{K}$-standard.
\end{lem}

\begin{proof}
By assumption, we observe that any complex $X$ in $\mathbf{K}^b(\mathcal{A})$ is isomorphic to a direct sum of stalk complexes. Let $(F, \omega)$ and $\theta_0$ be as in Definition \ref{defn:1}.  We set
$$\theta_{\Sigma^n(A)}=\Sigma^n((\theta_0)_A)\circ \omega^n_A\colon F(\Sigma^n A)\longrightarrow \Sigma^n (A)$$
for any $A\in \mathcal{A}$ and $n\in \mathbb{Z}$. By the additivity, $\theta_X\colon F(X)\rightarrow X$ is defined for any complex $X$; compare Lemma \ref{lem:twofun}. This yields the required extension of $\theta_0$, which is obviously unique.
\end{proof}

For a Krull-Schmidt category $\mathcal{A}$,  we denote by ${\rm ind}\mathcal{A}$  a complete set of representatives of isomorphism classes of indecomposable objects.

The following notion is slightly generalized from \cite{AR}; see also \cite{Chen}. A Krull-Schmidt category $\mathcal{A}$ is called an \emph{Orlov category} provided that the endomorphism ring of  each indecomposable object is a division ring and  that there is a degree function ${\rm deg}\colon {\rm ind}\mathcal{A}\rightarrow \mathbb{Z}$ satisfying ${\rm Hom}_\mathcal{A}(S, S')=0$ for any non-isomorphic $S, S'\in {\rm ind}\mathcal{A}$ with ${\rm deg} S\leq {\rm deg} S'$.

The following basic result is due to \cite[Section 4]{AR}.

\begin{prop}\label{prop:AR}
 Let $\mathcal{A}$ be an Orlov category. Then $\mathcal{A}$ is strongly $\mathbf{K}$-standard.
\end{prop}

\begin{proof}
 Let $(F, \omega)$ and $\theta_0$ be as in Definition \ref{defn:1}. Then $F|_\mathcal{A}$ is automatically homogeneous in the sense of \cite[Definition 4.1]{AR}. Then the existence of the extension $\theta$ follows from \cite[Theorem 4.7]{AR}, whose uniqueness follows from the commutative diagram (4.10) and Lemma 4.5(2) in \cite{AR}.

 In particular, the homomorphism (\ref{equ:CK}) for $\mathcal{A}$ is an isomorphism. This can  also  be deduced from  \cite[Proposition 2.2(ii)]{Chen}.
\end{proof}

\begin{exm}\label{exm:1}
{\rm  Let $k$ be a commutative artinian ring, and let $A$ be an artin $k$-algebra. Denote by $A\mbox{-proj}$ the category of finitely generated projective $A$-modules. Then $A\mbox{-proj}$ is an Orlov category if and only if $A$ is a triangular algebra, that is, the Ext-quiver of $A$ has no oriented cycles. For the statement,  the ``only if" part is clear, and the ``if" part is contained in \cite[Lemma 2.1]{Chen}.}
\end{exm}

\section{$\mathbf{D}$-standard abelian categories and standard equivalences}

In this section, we introduce the notions of a $\mathbf{D}$-standard abelian category and a strongly $\mathbf{D}$-standard abelian category. These are analogous to the ones in the previous section. The relation to standard derived equivalences is studied.

\subsection{$\mathbf{D}$-standard abelian categories}

Let $k$ be a commutative ring. Throughout, $\mathcal{A}$ is a $k$-linear abelian category.

\begin{defn}\label{defn:2}
We say that $\mathcal{A}$ is  \emph{$\mathbf{D}$-standard} (over $k$) provided that the following holds: given any $k$-linear triangle autoequivalence $(F, \omega)\colon \mathbf{D}^b(\mathcal{A})\rightarrow \mathbf{D}^b(\mathcal{A})$  satisfying $F(\mathcal{A})\subseteq \mathcal{A}$ and any natural isomorphism $\theta_0\colon F|_\mathcal{A}\rightarrow {\rm Id}_\mathcal{A}$, there is a natural transformation $\theta\colon (F, \omega)\rightarrow {\rm Id}_{\mathbf{D}^b(\mathcal{A})}$ of triangle functors which extends $\theta_0$.

The category $\mathcal{A}$ is said to be \emph{strongly $\mathbf{D}$-standard} (over $k$) if furthermore the above extension $\theta$ is always unique. \hfill $\square$
\end{defn}

We mention that the extension $\theta$ is necessarily an isomorphism. The following lemmas are analogous to Lemmas \ref{lem:K-stan}, \ref{lem:stan-center},  and \ref{lem:Kdef}. We omit the proofs.

\begin{lem}\label{lem:D-stan}
Let $\mathcal{A}$ be a $k$-linear abelian category. Then the following statements are equivalent.
\begin{enumerate}
\item  The abelian category $\mathcal{A}$ is $\mathbf{D}$-standard.
 \item For any $k$-linear pseudo-identity $(F, \omega)$ on $\mathbf{D}^b(\mathcal{A})$, there is a natural isomorphism $\eta\colon (F, \omega)\rightarrow {\rm Id}_{\mathbf{D}^b(\mathcal{A})}$ of triangle functors such that $\eta|_\mathcal{A}$ is the identity.
     \item Any $k$-linear pseudo-identity $(F, \omega)$ on $\mathbf{D}^b(\mathcal{A})$ is isomorphic to ${\rm Id}_{\mathbf{D}^b(\mathcal{A})}$, as triangle functors. \hfill $\square$
         \end{enumerate}
\end{lem}

\begin{lem}\label{lem:Dcenter}
Let $\mathcal{A}$ be a $k$-linear abelian category. Then it is strongly $\mathbf{D}$-standard if and only if it is $\mathbf{D}$-standard and the homomorphism  ${\rm res}\colon Z_\vartriangle(\mathbf{D}^b(\mathcal{A}))\rightarrow Z(\mathcal{A})$ in (\ref{equ:CD}) is an isomorphism. \hfill $\square$
\end{lem}

\begin{lem}
Let $\mathcal{A}$ be a $\mathbf{D}$-standard abelian category. Then the following statements hold.
\begin{enumerate}
\item Let $(F, \omega)\colon \mathbf{D}^b(\mathcal{A})\rightarrow \mathbf{D}^b(\mathcal{A})$ be a triangle autoequivalence with $F(\mathcal{A})\subseteq \mathcal{A}$. Then there is an isomorphism $(F, \omega)\stackrel{\sim}\longrightarrow \mathbf{D}^b(F|_\mathcal{A})$ of triangle functors.
    \item Assume further that $\mathcal{A}$ is strongly $\mathbf{D}$-standard. Let $F_1, F_2\colon \mathcal{A}\rightarrow \mathcal{A}$ be two autoequivalences, which are isomorphic. Then any natural transformation $\mathbf{D}^b(F_1)\rightarrow \mathbf{D}^b(F_2)$ of triangle functors is of the form $\mathbf{D}^b(\eta)$ for a unique natural transformation $\eta\colon F_1\rightarrow F_2$. \hfill $\square$
    \end{enumerate}
    \end{lem}

The following fact is essentially contained in the argument of \cite[2.16.4]{Or}.

\begin{prop}\label{prop:D-par-nat}
Let $(F, \omega)\colon \mathbf{D}^b(\mathcal{A})\rightarrow \mathbf{D}^b(\mathcal{A})$ be a triangle autoequivalence satisfying $F(\mathcal{A})\subseteq \mathcal{A}$. Assume that $\theta_0\colon F|_\mathcal{A}\rightarrow {\rm Id}_\mathcal{A}$ is a natural isomorphism, and that $\xi\colon A\rightarrow \Sigma^n(B)$ is a morphism in $\mathbf{D}^b(\mathcal{A})$ for $A, B\in \mathcal{A}$ and $n\geq 0$. Then we have $$\xi \circ (\theta_0)_A=\Sigma^n((\theta_0)_B)\circ \omega^n_B\circ F(\xi).$$
\end{prop}

\begin{proof}
The case of $n=0$ follows from the naturality of $\theta_0$. It suffices to prove the result for the case $n=1$. The general case follows by induction, once we observe the following fact: if $n> 1$, there exist an object $C\in \mathcal{A}$ and two morphisms $\xi_1\colon A\rightarrow \Sigma^{n-1}(C)$ and $\xi_2\colon C\rightarrow \Sigma(B)$ satisfying $\xi=\Sigma^{n-1}(\xi_2)\circ \xi_1$.

We assume that $n=1$. There is a short exact sequence $0\rightarrow B\stackrel{f}\rightarrow E\stackrel{g}\rightarrow A\rightarrow 0$ in $\mathcal{A}$,  which fits into an exact triangle $B\stackrel{f}\rightarrow E\stackrel{g}\rightarrow A\stackrel{\xi}\rightarrow \Sigma(B)$. The following commutative diagram between short exact sequences
\[\xymatrix{
0\ar[r] & F(B) \ar[d]^-{(\theta_0)_B} \ar[r]^-{F(f)} & F(E) \ar[d]^-{(\theta_0)_E} \ar[r]^-{F(g)} & F(A) \ar[d]^-{(\theta_0)_A}\ar[r] & 0\\
0\ar[r] & B \ar[r]^-{f} & E \ar[r]^-{g} & A \ar[r] & 0
}\]
induces a commutative diagram between exact triangles
\[\xymatrix{
 F(B) \ar[d]^-{(\theta_0)_B} \ar[r]^-{F(f)} & F(E) \ar[d]^-{(\theta_0)_E} \ar[r]^-{F(g)} & F(A) \ar[d]^-{(\theta_0)_A}\ar[r]^-{\omega_B\circ F(\xi)} &  \Sigma F(B) \ar[d]^{\Sigma((\theta_0)_B)}\\
  B \ar[r]^-{f} & E \ar[r]^-{g} & A \ar[r]^-{\xi} & \Sigma(B).
}\]
 Then we are done with $\xi\circ (\theta_0)_A=\Sigma((\theta_0)_B)\circ \omega_B\circ F(\xi)$.
\end{proof}

In view of Theorem \ref{thm:2} below, the following result extends \cite[Theorem 1.8]{MY}.

\begin{cor}
Let $\mathcal{A}$ be a $k$-linear abelian category which is hereditary. Then $\mathcal{A}$ is strongly $\mathbf{D}$-standard.
\end{cor}

\begin{proof}
 Assume that $(F, \omega)$ and $\theta_0$ are as above. For $A\in \mathcal{A}$ and $n\in \mathbb{Z}$,  we define $\theta_{\Sigma^n(A)}\colon F\Sigma^n (A)\rightarrow \Sigma^n(A)$ to be $\Sigma^n((\theta_0)_A)\circ \omega^n_A$. Proposition \ref{prop:D-par-nat} implies that for any morphism $\xi\colon\Sigma^n(A)\rightarrow \Sigma^m(B)$, we have
$$\xi\circ  \theta_{\Sigma^n(A)}=\theta_{\Sigma^m(B)}\circ F(\xi). $$
Here, we implicitly use the fact that $\omega^m_B=\Sigma^n(\omega_B^{m-n})\circ \omega^n_{\Sigma^{m-n}(B)}$. Since $\mathcal{A}$ is hereditary,  each complex $X$ in $\mathbf{D}^b(\mathcal{A})$ is isomorphic to $\oplus_{n\in \mathbb{Z}} \Sigma^{-n}(H^n(X))$. By Lemma \ref{lem:twofun}, we obtain a natural isomorphism $\theta\colon F\rightarrow {\rm Id}_{\mathbf{D}^b(\mathcal{A})}$; it is a natural isomorphism between triangle functors. This is the required extension of $\theta_0$, which is uniquely determined by $\theta_0$.
\end{proof}

The following notion is due to \cite{Or}. Recall that a sequence $\{P_i\}_{i\in \mathbb{Z}}$ of objects in $\mathcal{A}$ is \emph{ample} provided that for each object $X$, there exists $i(X)\in \mathbb{Z}$ such that for any $i\leq i(X)$, the following conditions hold:
\begin{enumerate}
\item there is an epimorphism $P_i^n\rightarrow X$ for some $n=n(i)$;
\item ${\rm Hom}_\mathcal{A}(X, P_i)=0$, and ${\rm Ext}_\mathcal{A}^j(P_i, X)=0$ for any $j>0$.
\end{enumerate}
We observe that if $\mathcal{A}$ has an ample sequence, there are no nonzero projective objects.

We have the following variant of \cite[Proposition 2.16]{Or}; see also \cite[Appendix]{BO}. We mention that the result plays an important role in the proof of the following famous theorem: any $k$-linear triangle equivalence between the bounded derived categories of coherent sheaves on smooth projective varieties is of \emph{Fourier-Mukai type}; see \cite{Or}.

\begin{prop}\label{prop:Orlov}
Let $\mathcal{A}$ and $\mathcal{B}$ be $k$-linear abelian categories with a triangle equivalence $G\colon \mathbf{D}^b(\mathcal{A})\rightarrow \mathbf{D}^b(\mathcal{B})$. Assume that the following conditions are satisfied:
\begin{enumerate}
\item $G(\mathcal{A})\cap \mathcal{B}$ contains an ample sequence of objects in $\mathcal{B}$;
\item for any object $X\in \mathcal{A}$, there is an epimorphism $P\rightarrow X$ with $P\in \mathcal{A}\cap G^{-1}(\mathcal{B})$. Here, we denote by $G^{-1}$ a quasi-inverse of $G$.
\end{enumerate}
Then $\mathcal{A}$ is strongly $\mathbf{D}$-standard.

In particular, an abelian category with an ample sequence of objects is strongly $\mathbf{D}$-standard.
\end{prop}

\begin{proof}
 Assume that $(F, \omega)$ and $\theta_0$ be as in Definition \ref{defn:2}. We observe that the triangle autoequivalence $GFG^{-1}$ on $\mathbf{D}^b(\mathcal{B})$ restricts to the identity endofunctor on $G(\mathcal{A})\cap \mathcal{B}$, via the isomorphism $G\theta_0G^{-1}$. Using the ample sequence contained in $G(\mathcal{A})\cap \mathcal{B}$, we apply \cite[Proposition 2.16]{Or} to obtain a unique isomorphism $\eta\colon GFG^{-1}\rightarrow {\rm Id}_{\mathbf{D}^b(\mathcal{B})}$ extending the isomorphism $G\theta_0 G^{-1}$ on $G(\mathcal{A})\cap \mathcal{B}$, where the uniqueness is proved in \cite[2.16.6]{Or}. Then the isomorphism $\theta=G^{-1}\eta G\colon F\rightarrow {\rm Id}_{\mathbf{D}^b(\mathcal{A})}$ extends  $\theta_0|_{\mathcal{A}\cap G^{-1}(\mathcal{B})}$. It indeed extends $\theta_0$ by (2) and a standard argument.

In more details, for any object $X\in \mathcal{A}$, we take an exact sequence $Q\stackrel{f}\rightarrow P\stackrel{g}\rightarrow X\rightarrow 0$ with $P, Q\in\mathcal{A}\cap G^{-1}(\mathcal{B})$. Then we have the following commutative exact diagram
\[\xymatrix{
F(Q) \ar[d]^-{\theta_Q}\ar[r]^-{F(f)} & F(P)\ar[d]^-{\theta_P} \ar[r]^-{F(g)} & F(X)\ar[d]^-{\theta_X}\ar[r] & 0\\
Q\ar[r]^-{f} & P\ar[r]^-{g} & X \ar[r] & 0.
}\]
Since $\theta_Q=(\theta_0)_Q$ and $\theta_P=(\theta_0)_P$, we infer that $\theta_X=(\theta_0)_X$.
\end{proof}

\subsection{Standard equivalences}

In this subsection, $k$ will be a field. For a finite dimensional $k$-algebra $A$, we denote by $A\mbox{-mod}$ the abelian category of finite dimensional left $A$-modules. Let $B$ be another finite dimensional $k$-algebra. The two algebras $A$ and $B$ are \emph{derived equivalent} (over $k$), provided that there is a $k$-linear triangle equivalence $(F, \omega)\colon \mathbf{D}^b(A\mbox{-mod})\rightarrow \mathbf{D}^b(B\mbox{-mod})$.

For any $B$-$A$-bimodule $_BM_A$, we always require that $k$ acts centrally. Recall that a bounded complex $_BX_A$ of $B$-$A$-bimodules is a \emph{two-sided tilting complex},  if the derived tensor functor $X\otimes_A^\mathbb{L}-\colon \mathbf{D}^b(A\mbox{-mod})\rightarrow \mathbf{D}^b(B\mbox{-mod})$ is an equivalence. We observe that $X\otimes_A^\mathbb{L}-$ is a triangle functor with a canonical connecting isomorphism.

Following \cite[Definition 3.4]{Ric}, a $k$-linear triangle equivalence $(F, \omega)\colon  \mathbf{D}^b(A\mbox{-mod})\rightarrow \mathbf{D}^b(B\mbox{-mod})$ is \emph{standard}, if it is isomorphic, as triangle functors,  to the derived tensor product $X\otimes^\mathbb{L}_A-$ for some two-sided tilting complex $X$. We mention that standard derived equivalences are closed under composition and quasi-inverse; for details, see \cite[6.5.2]{Zim}.

For a $k$-algebra automorphism $\sigma$ on $A$, we denote by $_\sigma A_1=A$ the $A$-$A$-bimodule with the left $A$-action twisted by $\sigma$; such a bimodule is a two-sided tilting complex. Recall that a $k$-linear autoequivalence $F\colon A\mbox{-mod}\rightarrow A\mbox{-mod}$ satisfying $F(A)\simeq A$ is necessarily isomorphic to the tensor functor ${_\sigma A_1}\otimes_A-$ for some automorphism $\sigma$.

In what follows, we suppress the connecting isomorphism for a triangle functor. The following result is essentially due to \cite[Corollary 3.5]{Ric}.

\begin{prop}\label{prop:factor}
Let $F\colon \mathbf{D}^b(A\mbox{-{\rm mod}})\rightarrow \mathbf{D}^b(B\mbox{-{\rm mod}})$ be a $k$-linear triangle equivalence. Then there exist  a pseudo-identity $F_1$ on $\mathbf{D}^b(A\mbox{-{\rm mod}})$  and a standard equivalence $F_2\colon \mathbf{D}^b(A\mbox{-{\rm mod}})\rightarrow \mathbf{D}^b(B\mbox{-{\rm mod}})$ such that $F$ is isomorphic to $F_2F_1$ as triangle functors.

 Such a factorization is unique. More precisely, if $F'_1$ is a pseudo-identity on $\mathbf{D}^b(A\mbox{-{\rm mod}})$ and $F'_2 \colon  \mathbf{D}^b(A\mbox{-{\rm mod}})\rightarrow \mathbf{D}^b(B\mbox{-{\rm mod}})$ is a standard equivalence such that $F$ is isomorphic to $F'_2F'_1$, then $F_i$ and $F'_i$ are isomorphic for $i=1, 2$.
\end{prop}

\begin{proof}
We observe that $F(A)$ is a one-sided tilting complex of $B$-modules. By \cite[Proposition 3.1]{Ric}, there is a two-sided tilting complex $X$ of $B$-$A$-bimodules with an isomorphism $X\rightarrow F(A)$ in $\mathbf{D}^b(B\mbox{-mod})$. Write $G$ for a quasi-inverse of $X\otimes_A^\mathbb{L}-$. It follows that $GF(A)\simeq A$ and then we have $GF(A\mbox{-mod})=A\mbox{-mod}$.

For the restricted equivalence $GF|_{A\mbox{-{\rm mod}}}\colon A\mbox{-mod}\rightarrow A\mbox{-mod}$, there exist an automorphism $\sigma$ on $A$ such that $GF|_{A\mbox{-mod}}$ is quasi-inverse to $_\sigma A_1\otimes -$. Denote by $H=\mathbf{D}^b({_\sigma A_1}\otimes-)\colon \mathbf{D}^b(A\mbox{-{\rm mod}})\rightarrow \mathbf{D}^b(A\mbox{-{\rm mod}})$ the induced equivalence, which is a standard equivalence. By Corollary \ref{cor:D-pseudo},  the composition $HGF$ is isomorphic to a pseudo-identity $F_1$ on $\mathbf{D}^b(A\mbox{-{\rm mod}})$, since its restriction to $A\mbox{-mod}$ is isomorphic to the identity functor. Set $F_2$ to be a quasi-inverse of $HG$, which is standard. Then we have the required factorization.

For the uniqueness, we observe that $F'_1F_1^{-1}$ is a pseudo-identity on $\mathbf{D}^b(A\mbox{-mod})$ and is isomorphic to $(F'_2)^{-1}F_2$. It follows that ${F'_1}{F_1}^{-1}$ is standard. Then we are done by Lemma \ref{lem:stan-pseudo} below.
\end{proof}

\begin{lem}\label{lem:stan-pseudo}
Let $F\colon \mathbf{D}^b(A\mbox{-{\rm mod}})\rightarrow \mathbf{D}^b(A\mbox{-{\rm mod}})$ be a pseudo-identity. Assume that $F$ is standard. Then there is a natural isomorphism $F\rightarrow {\rm Id}_{\mathbf{D}^b(A\mbox{-{\rm mod}})}$ of triangle functors, whose restriction to $A\mbox{-{\rm mod}}$ is the identity.
\end{lem}

\begin{proof}
Assume that $F\simeq X\otimes_A^\mathbb{L}-$ for a two-sided tilting complex  $X$ of $A$-$A$-bimodules. By $X\otimes_A^\mathbb{L}A\simeq A$, we infer that $_AX_A$ is isomorphic to  a stalk complex concentrated in  degree zero. So we view $X$ as an $A$-$A$-bimodule, where $_AX$ is isomorphic to $_AA$ as a left $A$-module.

Since $X\otimes_A^\mathbb{L} M\simeq M$ for any $A$-module $M$, we infer that $X_A$ is projective as a right $A$-module. Hence, we have $F\simeq \mathbf{D}^b(X\otimes_A-)$, whose restriction to $A\mbox{-mod}$ is the tensor functor $X\otimes_A-$. Recall that the restriction $F$ to $A\mbox{-mod}$ is the identity functor. It follows that $X$ is isomorphic to the regular bimodule $_AA_A$. Therefore, $\mathbf{D}^b(X\otimes_A-)$ is isomorphic to the identity functor on $\mathbf{D}^b(A\mbox{-mod})$.

In summary, we have proved that $F$ is isomorphic to ${\rm Id}_{\mathbf{D}^b(A\mbox{-{\rm mod}})}$, as triangle functors. By Lemma \ref{lem:D-pseudo}, we are done.
\end{proof}

The following result actually motivates our study of $\mathbf{D}$-standard categories.

\begin{thm}\label{thm:2}
Let $A$ be a finite dimensional $k$-algebra. Then the following statements are equivalent.
\begin{enumerate}
\item The module category $A\mbox{-{\rm mod}}$ is $\mathbf{D}$-standard over $k$.
\item Any $k$-linear derived equivalence $\mathbf{D}^b(A\mbox{-{\rm mod}})\rightarrow \mathbf{D}^b(B\mbox{-{\rm mod}})$ is standard.
    \item Any $k$-linear derived equivalence $\mathbf{D}^b(A\mbox{-{\rm mod}})\rightarrow \mathbf{D}^b(A\mbox{-{\rm mod}})$ is standard.
\end{enumerate}
\end{thm}

\begin{proof}
By combining Proposition \ref{prop:factor} and Lemma \ref{lem:D-stan}, we have $(1)\Rightarrow (2)$. The implication $(2)\Rightarrow (3)$ is clear. For $(3)\Rightarrow (1)$, we apply Lemma \ref{lem:stan-pseudo} to obtain that any pseudo-identity on $\mathbf{D}^b(A\mbox{-{\rm mod}})$ is isomorphic to the identity functor. Then we are done by Lemma \ref{lem:D-stan}.
\end{proof}

It is an open question whether all $k$-linear derived equivalences are standard; see the remarks after \cite[Definition 3.4]{Ric}. In view of Theorem \ref{thm:2}, an affirmative answer is equivalent to  the following conjecture.

\begin{conj}\label{conj:main}
For any finite dimensional $k$-algebra $A$, the module category $A\mbox{-{\rm mod}}$ is $\mathbf{D}$-standard over $k$.
\end{conj}

On the other hand, it would be nice to have an explicit example of non-$\mathbf{D}$-standard abelian categories. We mention the work \cite{PV}, where the above open question is treated using filtered triangulated categories.

By the following result, it suffices to verify Conjecture \ref{conj:main} up to derived equivalences.

\begin{lem}\label{lem:conj}
Let $A$ and $B$ be two algebras which are derived equivalent.  Then $A\mbox{-{\rm mod}}$ is (resp. strongly) $\mathbf{D}$-standard if and only if $B\mbox{-{\rm mod}}$ is (resp. strongly) $\mathbf{D}$-standard.
\end{lem}

\begin{proof}
Assume that $A\mbox{-mod}$ is $\mathbf{D}$-standard. Take a standard derived  equivalence $G\colon \mathbf{D}^b(A\mbox{-mod})\rightarrow \mathbf{D}^b(B\mbox{-mod})$. For any triangle autoequivalence $F$ on $\mathbf{D}^b(B\mbox{-mod})$, in view of Theorem \ref{thm:2}(2),  we have that the composition $FG$ is standard. It follows that $F$ is standard, since it is isomorphic to $(FG)G^{-1}$, as the composition of two standard equivalences. This shows that $B\mbox{-mod}$ is $\mathbf{D}$-standard by Theorem \ref{thm:2}(3).

If  $A\mbox{-mod}$ is strongly $\mathbf{D}$-standard, the homomorphism (\ref{equ:CD}) for $A\mbox{-mod}$ is an isomorphism. By \cite[Proposition 9.2]{Ric89}, the centers $Z(A\mbox{-mod})$ and $Z(B\mbox{-mod})$ are isomorphic, since they are isomorphic to the centers $Z(A)$ and $Z(B)$ of the algebras, respectively. The triangle centers $Z_\vartriangle(\mathbf{D}^b(A\mbox{-mod}))$ and $Z_\vartriangle(\mathbf{D}^b(B\mbox{-mod}))$ are also isomorphic. By a dimension argument, the homomorphism (\ref{equ:CD}) for $B\mbox{-mod}$, which is always surjective,  is necessarily  an isomorphism. By Lemma \ref{lem:Dcenter}, the module category $B\mbox{-mod}$ is strongly $\mathbf{D}$-standard. Then we are done.
\end{proof}

In view of Lemma \ref{lem:conj} and Proposition \ref{prop:Orlov}, it is natural to ask the following general question: for two $k$-linear abelian categories $\mathcal{A}$ and $\mathcal{B}$  which are derived equivalent such that $\mathcal{A}$ is (\emph{resp}. strongly) $\mathbf{D}$-standard, so is $\mathcal{B}$?

Let us recall from \cite{MY, Mina, Chen} the cases where Conjecture \ref{conj:main} is actually confirmed. We mention that the case of a canonical algebra is studied in \cite[Lemma 6.6]{BKL}.

Following \cite[Definition 4.1]{Mina}, a finite dimensional algebra $A$ is \emph{Fano} (resp. \emph{anti-Fano}), if $A$ has finite global dimension and for some natural number $d$,  $\Sigma^{-d}(DA)$, as a two-sided tilting complex of $A$-$A$-bimodules, is anti-ample (\emph{resp}. ample) in the sense of \cite[Definition 3.4]{Mina}. Here, $DA={\rm Hom}_k(A, k)$.

\begin{prop}
Let $A$ be a finite dimensional $k$-algebra. Then $A\mbox{-{\rm mod}}$ is strongly $\mathbf{D}$-standard provided that $A$ is derived equivalent to a triangular algebra or a (anti-)Fano algebra.
\end{prop}

\begin{proof}
In view of Lemma \ref{lem:conj}, we may assume that $A$ is triangular or (anti-)Fano. In the first case, the category $A\mbox{-proj}$ is strongly $\mathbf{K}$-standard; see Example \ref{exm:1}. We just apply Theorem \ref{thm:1} below; compare \cite[Theorem 1.1]{Chen}.

The second case follows from \cite[Theorem 4.5]{Mina}, where the uniqueness of the extension of $\theta_0$ in Definition \ref{defn:2} follows from the uniqueness established in \cite[2.16.6]{Or}; compare the last paragraph  in the proof of \cite[Theorem 4.6]{Mina}.
\end{proof}

\section{$\mathbf{K}$-standardness versus $\mathbf{D}$-standardness}

Let $k$ be a commutative ring. For a $k$-linear abelian category $\mathcal{A}$ with enough projectives, we denote by $\mathcal{P}$ the full subcategory formed by projective objects. The main result shows that the $\mathbf{K}$-standardness of $\mathcal{P}$ implies the $\mathbf{D}$-standardness of $\mathcal{A}$. This seems to be useful to study Conjecture \ref{conj:main}.

\begin{thm}\label{thm:1}
Let $\mathcal{A}$ be a $k$-linear abelian category with enough projective objects. Denote by $\mathcal{P}$ the full subcategory of projective objects.  Assume that $\mathcal{P}$ is $\mathbf{K}$-standard. Then $\mathcal{A}$ is $\mathbf{D}$-standard. In this case, $\mathcal{A}$ is strongly $\mathbf{D}$-standard if and only if $\mathcal{P}$ is strongly $\mathbf{K}$-standard.
\end{thm}

\begin{proof}
The last statement follows from Lemmas \ref{lem:stan-center}, \ref{lem:Dcenter} and the commutative diagram (\ref{equ:center}),  whose rows are both isomorphisms.

To show that $\mathcal{A}$ is $\mathbf{D}$-standard, we assume that $(F, \omega)$ is a pseudo-identity on $\mathbf{D}^b(\mathcal{A})$. We view $\mathbf{K}^b(\mathcal{P})$ as a triangulated subcategory of $\mathbf{D}^b(\mathcal{A})$. Then $(F, \omega)$ restricts to a pseudo-identity $(F', \omega)$ on $\mathbf{K}^b(\mathcal{P})$, whose connecting isomorphism is inherited from $F$. Since $\mathcal{P}$ is $\mathbf{K}$-standard, there is a natural isomorphism $\delta\colon (F', \omega)\rightarrow {\rm Id}_{\mathbf{K}^b(\mathcal{P})}$,  which satisfies that $\delta_P={\rm Id}_P$ for any object $P\in \mathcal{P}$.

For a bounded complex $P$ of projective objects and $n\geq 0$, we claim that
\begin{align}\label{equ:omega}
\Sigma^{n+1}(\delta_P)\circ \omega^{n+1}_P=\Sigma(\delta_{\Sigma^n(P)})\circ \omega_{\Sigma^n(P)}.
\end{align}
Indeed, since $\delta$ is a morphism of triangle functors, we have  $\delta_{\Sigma(P)}= \Sigma(\delta_P)\circ \omega_P$. Using induction, we have  $\delta_{\Sigma^{n}(P)}=\Sigma^{n}(\delta_P)\circ \omega^{n}_P$. Now the claim follows from the identity $\omega^{n+1}_P= \Sigma(\omega^{n}_P)\circ \omega_{\Sigma^{n}(P)}$.

Take an arbitrary complex $X$ in $\mathbf{D}^b(\mathcal{A})$. We may assume that $X$ is isomorphic to a complex of the form
\begin{align}\label{equ:X}
\cdots \rightarrow 0\rightarrow A \stackrel{\partial}\rightarrow P^{1-n} \rightarrow P^{2-n}\rightarrow \cdots \rightarrow P^{m-1}\rightarrow P^m \rightarrow 0 \rightarrow \cdots
\end{align}
 with each $P^i$ projective, $m, n\geq 0$,  $A\in \mathcal{A}$  and $\partial$ a monomorphism. Therefore, we have an exact triangle
$$P\stackrel{\iota}\longrightarrow X \stackrel{p}\longrightarrow \Sigma^n(A)\stackrel{h}\longrightarrow \Sigma(P),$$
where $\iota$ is given by the inclusion of complexes and $p$ is the projection. The complex $P$ lies in $\mathbf{K}^b(\mathcal{P})$.  The chain map $h$ is given by the map $\partial\colon A \rightarrow P^{1-n}$. More precisely, we have $\Sigma(c)\circ h=\Sigma^n(\partial)$, where $c\colon P\rightarrow \Sigma^{n-1}(P^{1-n})$ denotes the canonical projection.

The following  observation will be used frequently.
\begin{align}\label{equ:useful}
{\rm Hom}_{\mathbf{D}^b(\mathcal{A})}(F\Sigma^n(A), X)={\rm Hom}_{\mathbf{D}^b(\mathcal{A})}(\Sigma^n(A), X)=0
\end{align}
Indeed, by the injectivity of the morphism $\partial$, $X$ is isomorphic to its good truncation $\tau_{\geq 1-n}(X)$. Then we are done by the standard $t$-structure in $\mathbf{D}^b(\mathcal{A})$; see \cite[Example 1.3.2(i)]{BBD}.

We claim that the following diagram commutes.
\[\xymatrix{
F\Sigma^n(A) \ar[d]_{\omega^n_A} \ar[rr]^-{\omega_P\circ F(h)} && \Sigma (FP) \ar[d]^-{\Sigma(\delta_P)}\\
\Sigma^n(A)\ar[rr]^-{h} && \Sigma(P)
}\]
Recall the canonical projection $c\colon P\rightarrow \Sigma^{n-1}(P^{1-n})$. We have
\begin{align*}
\Sigma(c)\circ h  \circ \omega^n_A & = \Sigma^n(\partial)\circ \omega^n_A\\
& = \Sigma^n F(\partial) \circ \omega^n_A\\
& =  \omega^n_{P^{1-n}} \circ F\Sigma^n(\partial)\\
& =\Sigma(\delta_{\Sigma^{n-1}(P^{1-n})} )\circ \omega_{\Sigma^{n-1}(P^{1-n})}\circ F\Sigma^n(\partial),
\end{align*}
where the second equality uses the pseudo-identity $F$, and the last uses (\ref{equ:omega}) applied to $P^{1-n}$.  On the other hand, we have
\begin{align*}
\Sigma(c)\circ \Sigma(\delta_P)\circ \omega_P\circ F(h) &= \Sigma (\delta_{\Sigma^{n-1}(P^{1-n})}) \circ \Sigma F(c)\circ \omega_P\circ F(h)\\
&= \Sigma (\delta_{\Sigma^{n-1}(P^{1-n})})\circ \omega_{\Sigma^{n-1}(P^{1-n})} \circ F\Sigma(c)\circ F(h)\\
&= \Sigma (\delta_{\Sigma^{n-1}(P^{1-n})}) \circ \omega_{\Sigma^{n-1}(P^{1-n})} \circ F\Sigma^n(\partial).
\end{align*}
We conclude that $\Sigma(c)\circ h \circ  \omega^n_A=
\Sigma(c)\circ \Sigma(\delta_P)\circ \omega_P\circ F(h)$. Then the claim follows from the following observation: the following map
$${\rm Hom}(F\Sigma^n(A), \Sigma(c))\colon {\rm Hom}(F\Sigma^n(A), \Sigma(P))\longrightarrow {\rm Hom}(F\Sigma^n(A), \Sigma^n(P^{1-n}))$$
is injective, where Hom means the Hom spaces in $\mathbf{D}^b(\mathcal{A})$. Indeed, the cone of $c$ is $\Sigma (\sigma_{\geq 2-n}P)$. Then the required injectivity follows from
$${\rm Hom}_{\mathbf{D}^b(\mathcal{A})}(F\Sigma^n(A), \Sigma (\sigma_{\geq 2-n}P))=0. $$

Applying the above claim and (TR3), we obtain the following commutative diagram between exact triangles
\[
\xymatrix{
F(P)\ar[r]^-{F(\iota)} \ar[d]^-{\delta_P} & F(X) \ar[r]^-{F(p)} \ar@{.>}[d]^-{\theta_X} & F\Sigma^n(A) \ar[r]^-{\omega_P\circ F(h)} \ar[d]^-{\omega^n_A}& \Sigma(FP) \ar[d]^-{\Sigma(\delta_P)}\\
P\ar[r]^-{\iota} & X \ar[r]^-{p} & \Sigma^n(A) \ar[r]^-{h} & \Sigma(P),
}\]
where the morphism $\theta_X$ is necessarily an isomorphism. By (\ref{equ:useful}) and \cite[Proposition 1.1.9]{BBD}, we infer that such a morphism $\theta_X$ is unique.

We have to show that the isomorphism $\theta_X$ is independent of the choice of the complex (\ref{equ:X}). Assume that $X$ is isomorphic to  another complex
\begin{align}\label{equ:Y}
\cdots \rightarrow 0\rightarrow B \stackrel{\partial'}\rightarrow Q^{1-r} \rightarrow Q^{2-r}\rightarrow \cdots \rightarrow Q^{s-1}\rightarrow Q^s \rightarrow 0 \rightarrow \cdots
\end{align}
with each $Q^j$ projective and $\partial'$ a monomorphism. Then we have the corresponding exact triangle
$$Q\stackrel{\iota'}\longrightarrow X \stackrel{p'}\longrightarrow \Sigma^r(B)\stackrel{h'}\longrightarrow \Sigma(Q)$$
and the commutative diagram,  which defines the isomorphism $\tilde{\theta}_X$.
\[
\xymatrix{
F(Q)\ar[r]^-{F(\iota')} \ar[d]^-{\delta_Q} & F(X) \ar[r]^-{F(p')} \ar@{.>}[d]^-{\tilde{\theta}_X} & F\Sigma^r(B) \ar[r]^-{\omega_Q\circ F(h')} \ar[d]^-{\omega^r_B}& \Sigma(FQ) \ar[d]^-{\Sigma(\delta_Q)}\\
Q\ar[r]^-{\iota'} & X \ar[r]^-{p'} & \Sigma^r(B) \ar[r]^-{h'} & \Sigma(Q)}
\]
We assume without loss of generality that $r\geq n$. By ${\rm Hom}_{\mathbf{D}^b(\mathcal{A})}(P, \Sigma^r (B))=0$ and \cite[Proposition 1.1.9]{BBD}, we have the following commutative diagram.
\[\xymatrix{
P \ar[r]^-{\iota} \ar@{.>}[d]^-{a}& X \ar[r]^-{p} \ar@{=}[d] & \Sigma^n(A) \ar@{.>}[d]^-{b} \ar[r]^-{h} & \Sigma(P) \ar@{.>}[d]^-{\Sigma(a)}\\
Q \ar[r]^-{\iota'} & X \ar[r]^-{p'} & \Sigma^r(B) \ar[r]^-{h'} & \Sigma(Q)
}\]
Then we have
\begin{align*}
\theta_X\circ F(\iota)&=\iota\circ \delta_P\\
                      &=\iota'\circ a\circ \delta_P\\
                      &=\iota'\circ \delta_Q\circ F(a)\\
                      &=\tilde{\theta}_X\circ F(\iota')\circ F(a)\\
                      &=\tilde{\theta}_X\circ F(\iota),
\end{align*}
where the third equality uses the naturality of $\delta$. We infer that $\theta_X-\tilde{\theta}_X$ factors through $F\Sigma^n(A)$.  Using ${\rm Hom}_{\mathbf{D}^b(\mathcal{A})}(F\Sigma^n(A), X)=0$ in (\ref{equ:useful}), we infer that $\theta_X=\tilde{\theta}_X$, as required.

To prove the naturality of $\theta$, we assume that $f\colon X\rightarrow Y$ is a morphism. We may assume that $X$ is isomorphic to the complex (\ref{equ:X}) and that $Y$ is isomorphic to the complex (\ref{equ:Y}). Moreover, we may assume that $r=n$ and that the morphism $f$ is given by a chain map between these complexes.  Consequently, using these assumptions, we have a commutative diagram.
\[\xymatrix{
P \ar[r]^-{\iota} \ar@{.>}[d]^-{e}& X \ar[r]^-{p} \ar[d]^-{f} & \Sigma^n(A) \ar@{.>}[d]^-{d} \ar[r]^-{h} & \Sigma(P) \ar@{.>}[d]^-{\Sigma(e)}\\
Q \ar[r]^-{\iota'} & Y \ar[r]^-{p'} & \Sigma^n(B) \ar[r]^-{h'} & \Sigma(Q)
}\]
Then using the same argument as above, we infer that
$$(f\circ \theta_X)\circ F(\iota)=(\theta_Y\circ F(f))\circ F(\iota).$$
Since $r=n$, we observe that ${\rm Hom}_{\mathbf{D}^b(\mathcal{A})}(F\Sigma^n(A), Y)=0$ as in (\ref{equ:useful}). We deduce  that $f\circ \theta_X=\theta_Y\circ F(f)$.

It remains to show that $\theta\colon (F, \omega)\rightarrow {\rm Id}_{\mathbf{D}^b(\mathcal{A})}$ is a natural transformation between triangle functors, that is, $\theta_{\Sigma(X)}=\Sigma(\theta_X)\circ \omega_X$ for each complex $X$. We observe that the following commutative diagram defines $\theta_{\Sigma(X)}$.
\[
\xymatrix{
F\Sigma(P)\ar[r]^-{F\Sigma(\iota)} \ar[d]^-{\delta_{\Sigma(P)}} & F\Sigma(X) \ar[r]^-{F\Sigma(p)} \ar@{.>}[d]^-{\theta_{\Sigma(X)}} & F\Sigma^{n+1}(A) \ar[rr]^-{-\omega_{\Sigma(P)}\circ F\Sigma(h)} \ar[d]^-{\omega^{n+1}_A}&& \Sigma(F\Sigma P) \ar[d]^-{\Sigma(\delta_{\Sigma(P)})}\\
\Sigma(P)\ar[r]^-{\Sigma(\iota)} & \Sigma(X) \ar[r]^-{\Sigma(p)} & \Sigma^{n+1}(A) \ar[rr]^-{-\Sigma(h)} && \Sigma^2(P)
}\]
Then we have
\begin{align*}
\theta_{\Sigma(X)} \circ F\Sigma(\iota) &= \Sigma(\iota)\circ \delta_{\Sigma(P)}\\
                                        &= \Sigma(\iota) \circ \Sigma(\delta_P)\circ \omega_P\\
                                        &= \Sigma(\iota\circ \delta_P) \circ \omega_P\\
                                        &=\Sigma(\theta_X\circ F(\iota))\circ \omega_P\\
                                        &=\Sigma(\theta_X)\circ \omega_X\circ F\Sigma(\iota),    \end{align*}
where the second equality uses the fact that $\delta$ is a natural transformation between triangle functors. It follows that $\theta_{\Sigma(X)}-\Sigma(\theta_X)\circ \omega_X$ factors through $F\Sigma^{n+1}(A)=\Sigma^{n+1}(A)$. However, by ${\rm Hom}_{\mathbf{D}^b(\mathcal{A})}(\Sigma^{n+1}(A), \Sigma(X))=0$ from (\ref{equ:useful}), we infer that $\theta_{\Sigma(X)}=\Sigma(\theta_X)\circ \omega_X$.  In view of Lemma \ref{lem:D-stan}, we are done with the whole proof.
\end{proof}

The following is a partial converse of Theorem \ref{thm:1}.

\begin{prop}
Let $\mathcal{A}$ be a $k$-linear abelian category with enough projectives. Denote by $\mathcal{P}$ the full subcategory of projective objects. Assume that each object has finite projective dimension. If $\mathcal{A}$ is $\mathbf{D}$-standard, then $\mathcal{P}$ is $\mathbf{K}$-standard.
\end{prop}

\begin{proof}
By the assumption, the obvious inclusion functor $\mathbf{K}^b(\mathcal{P})\rightarrow \mathbf{D}^b(\mathcal{A})$ is an equivalence. Let $F$ be a pseudo-identity on $\mathbf{K}^b(\mathcal{P})$. Then $F$ induces a triangle autoequivalence $F'$ on $\mathbf{D}^b(\mathcal{A})$ satisfying that $F'(X)\simeq X$ and $F'|_\mathcal{P}$ is isomorphic to the identity functor. It follows with a standard argument that $F'|_\mathcal{A}$ is also isomorphic to the identity functor. Consequently, by Corollary \ref{cor:D-pseudo} $F'$ is isomorphic to a pseudo-identity on $\mathbf{D}^b(\mathcal{A})$. By the $\mathbf{D}$-standardness of $\mathcal{A}$, we infer that $F'$ is isomorphic to the identity functor on  $\mathbf{D}^b(\mathcal{A})$. It follows that $F$ is isomorphic to the identity functor ${\rm Id}_{\mathbf{K}^b(\mathcal{P})}$ as a triangle functor. Then we are done by Lemma~\ref{lem:K-stan}.
\end{proof}

\section{Two examples}

In this section, we provide two examples of algebras whose module categories are $\mathbf{D}$-standard. In other words, Conjecture \ref{conj:main} is confirmed for these examples.

Throughout, $k$ will be a field. For a finite dimensional algebra $A$, we denote by $A\mbox{-proj}$ the category of finitely generated projective $A$-modules.

\subsection{The dual numbers} Let $A=k[\varepsilon]$ be the algebra of \emph{dual numbers}, that is, $A=k1_A\oplus k\varepsilon$ with $\varepsilon^2=0$.

\begin{thm}\label{thm:exam1}
Let $A=k[\varepsilon]$ be the algebra of dual numbers. Then the category $A\mbox{-{\rm proj}}$ is $\mathbf{K}$-standard, but not strongly $\mathbf{K}$-standard. Consequently, the module category $A\mbox{-{\rm mod}}$ is $\mathbf{D}$-standard, but not strongly $\mathbf{D}$-standard.
\end{thm}

The structure of $\mathbf{K}^b(A\mbox{-proj})$ is well known; see \cite{Ku, KY}. For any $n\leq m$, we denote by $X_{n, m}$ the following complex
\begin{align*}
\cdots \rightarrow 0 \rightarrow A \rightarrow A \rightarrow \cdots \rightarrow A\rightarrow A \rightarrow 0\rightarrow \cdots
\end{align*}
where the nonzero components start at degree $n$ and end at degree $m$. The unnamed arrow $A\rightarrow A$ is the morphism induced by the multiplication of $\varepsilon$. In particular, $X_{n, n}=\Sigma^{-n}(A)$ is the stalk complex concentrated at degree $n$.

For $n\leq m$, we denote by $i_{n, m}\colon X_{n, m}\rightarrow X_{n-1, m}$ the inclusion map,  and by $\pi_{n, m}\colon X_{n, m}\rightarrow X_{n, m-1}$ the canonical projection if further $n<m$. For $n\leq m\leq l$, we denote by $c_{n, m, l}\colon X_{n, m}\rightarrow X_{m, l}$ the following chain map
\[\xymatrix@C=14pt{
0\ar[r] & A \ar[r] & \cdots \ar[r] & A\ar[r] & A \ar[d]\ar[r] & 0   \\
 & & &   0 \ar[r]  &  A \ar[r] & \cdots \ar[r] & A\ar[r] & A \ar[r] & 0.
}\]
Here, as above, the unnamed arrows $A\rightarrow A$ denote the morphism given by the multiplication of $\varepsilon$. We observe the following exact triangle
\begin{align}\label{equ:tri1}
X_{m, m}\xrightarrow{i_{n+1, m}\circ \cdots \circ i_{m-1, m}\circ i_{m, m}} X_{n,m}\xrightarrow{\pi_{n, m}}  X_{n, m-1}\xrightarrow{c_{n, m-1, m-1}}  X_{m-1, m-1}.
\end{align}

We denote by $\Delta_{n, m}$ the following composition
$$X_{n, m}\xrightarrow{c_{n, m, m}}  X_{m, m}\stackrel{i_{m, m}}\longrightarrow X_{m-1, m}\longrightarrow \cdots \longrightarrow X_{n+1, m}\xrightarrow{i_{n+1, m}}  X_{n, m}.$$
We set $\Delta_{n,n}=c_{n, n, n}$.

The following results are well known; compare \cite[Lemma 5.1]{KY}.

\begin{lem}\label{lem:fact1}
The following facts hold.
\begin{enumerate}
\item $\{X_{n, m}\; |\; n\leq m\}$ is a complete set of representatives of the isomorphism classes of indecomposable objects in $\mathbf{K}^b(A\mbox{-{\rm proj}})$.
\item For $n\leq m$ and $1\leq r$,  we have ${\rm Hom}(X_{n, m}, X_{n-r, m})=k(i_{n-r+1, m}\circ \cdots \circ i_{n-1, m}\circ i_{n, m})$.
    \item For $n<m$ and $1\leq r\leq m-n$, we have  ${\rm Hom}(X_{n, m}, X_{n, m-r})=k(\pi_{n, m-r+1}\circ \cdots \circ \pi_{n, m-1}\circ  \pi_{n, m})$.
     \item For $n\leq m\leq l$, we have  ${\rm Hom}(X_{n, m}, X_{m, l})=kc_{n, m, l}$ if $n<m$ or $m<l$.
    \item For $n\leq m$, we have ${\rm Hom}(X_{n, m}, X_{n, m})=k{\rm Id}_{X_{n, m}}\oplus k\Delta_{n, m}$, where the endomorphism $\Delta_{n, m}$ is almost vanishing.
    \item The morphisms $\{i_{n, m}, \pi_{n, m}, c_{n,m, l}\; |\; n\leq m\leq l\}$ span the $k$-linear category $\mathbf{K}^b(A\mbox{-{\rm proj}})$. \hfill $\square$
\end{enumerate}
\end{lem}

\begin{lem}\label{lem:lambda}
Let $(F, \omega)\colon \mathbf{K}^b(A\mbox{-{\rm proj}})\rightarrow \mathbf{K}^b(A\mbox{-{\rm proj}})$ be a pseudo-identity. Assume that $F(i_{n,m})=\lambda_{n, m}i_{n, m}$ and $F(\pi_{n, m})=\mu_{n, m}\pi_{n, m}$ for any $n\leq m$, where $\lambda_{n, m}$ and $\mu_{n, m}$ are nonzero scalars. Then the following statements hold.
\begin{enumerate}
\item For $n\leq m\leq l$, we have
$$F(c_{n,m, l})=(\lambda_{n+1, m} \cdots \lambda_{m-1, m}\lambda_{m, m}  \cdot \mu_{m, m+1} \cdots \mu_{m, l-1} \mu_{m ,l})^{-1} c_{n, m, l}.$$
\item Assume for each $m\in \mathbb{Z}$ that $\omega_{(X_{m, m})}=\Sigma(a_{m}+b_{m}\Delta_{m, m})$ with $a_{m}, b_m\in k$ and $a_{m}\neq 0$. Then for $n<m$ we have
    \begin{align}\label{equ:lambda}
    \lambda_{n+1, m}\cdots \lambda_{m-1, m}\lambda_{m, m}\mu_{n, m}a_m=\lambda_{n+1, m-1}\cdots \lambda_{m-2, m-1}\lambda_{m-1, m-1}.
    \end{align}
\end{enumerate}
\end{lem}

For (1), we observe that if $n=m$, the coefficients $\lambda$'s do not appear; if $m=l$, $\mu$'s do not appear.

\begin{proof}
(1) We denote by $\phi\colon A\rightarrow A$ the morphism induced by the multiplication of $\varepsilon$. Then we have that  $F\Sigma^m(\phi)=\Sigma^m(\phi)$ for $m\in \mathbb{Z}$. In view of Lemma \ref{lem:fact1}(4), we have that $F(c_{n, m, l})$ equals $c_{n, m, l}$ up to a nonzero scalar. We observe that
$$\Sigma^{-m}(\phi)=(\pi_{m, m+1}\circ \cdots \circ \pi_{m, l-1}\circ \pi_{m, l})\circ c_{n, m, l}\circ  (i_{n+1, m}\circ \cdots \circ i_{m-1, m}\circ i_{m, m}).$$
Applying $F$ to both sides and using the claim above, we are done.

(2) We apply the triangle functor $(F, \omega)$ to the exact triangle (\ref{equ:tri1}). Then the three morphisms in the triangle change up to  nonzero scalars. Here, we observe that
$$\omega_{(X_{m, m})}\circ F(c_{n, m-1, m-1})=a_{m} (\lambda_{n+1, m-1}\cdots \lambda_{m-2, m-1}\lambda_{m-1, m-1})^{-1}c_{n, m-1, m-1}.$$
The resulted triangle is still exact. Then we are done by Proposition \ref{prop:almvan}.
\end{proof}

\vskip 5pt

\noindent \emph{Proof of Theorem \ref{thm:exam1}.} In the proof, we put $\mathcal{A}=A\mbox{-{\rm proj}}$  and $\mathcal{T}=\mathbf{K}^b(A\mbox{-{\rm proj}})$.

 Let $(F, \omega)$ be a pseudo-identity on $\mathcal{T}$. As above, we assume that $F(i_{n,m})=\lambda_{n, m}i_{n, m}$ and $F(\pi_{n, m})=\mu_{n, m}\pi_{n, m}$ for any $n\leq m$. Assume for each $m\in \mathbb{Z}$ that $\omega_{(X_{m, m})}=\Sigma(a_{m}+b_{m}\Delta_{m, m})$ with $a_{m}, b_m\in k$ and $a_{m}\neq 0$. We take nonzero scalars $c_m$ such that $c_0=1$ and $a_m=c_{m-1}(c_m)^{-1}$. For each complex $X$, we choose an isomorphism $$\delta_X\colon F(X)=X\longrightarrow X=F'(X)$$
  such that $\delta_{\Sigma^m(X)}=c_{-m}{\rm Id}_{\Sigma^m(X)}$ for any $X\in \mathcal{A}$ and $m\in \mathbb{Z}$,  and that
 $$\delta_{(X_{n, m})}=c_m (\lambda_{n+1, m}\cdots \lambda_{m-1, m}\lambda_{m, m})^{-1} {\rm Id}_{X_{n, m}}$$
for any $n< m$. By assumption, we observe that $\delta_{(X_{m, m})}=c_m {\rm Id}_{X_{m, m}}$.

Using these $\delta_X$'s as the adjusting isomorphisms, we obtain a new triangle functor $(F', \omega')$ such that $\delta\colon (F, \omega)\rightarrow (F', \omega')$ is an isomorphism. We observe that $F'$ is also a pseudo-identity. We claim that $F'(i_{n, m})=i_{n, m}$ and $F'(\pi_{n, m})=\pi_{n, m}$. Indeed, the claim is equivalent to the following identities:
$$i_{n, m}\circ \delta_{(X_{n, m})}=\delta_{(X_{n-1, m})} \circ F(i_{n, m}), \quad \pi_{n, m}\circ \delta_{(X_{n, m})}=\delta_{(X_{n, m-1})}\circ F(\pi_{n, m}). $$
The left identity follows from the definition of these isomorphisms $\delta_X$'s, and the right one follows from (\ref{equ:lambda}) and the fact that $a_m=c_{m-1}(c_m)^{-1}$.

Applying Lemma \ref{lem:lambda}(1) to $F'$, we infer that $F'(c_{n, m, l})=c_{n, m, l}$. Since these morphisms span $\mathcal{T}$, by Lemma \ref{lem:iso-gen} there is a unique natural isomorphism $\delta'\colon F'\rightarrow {\rm Id}_\mathcal{T}$ such that $\delta'_{(X_{n, m})}={\rm Id}_{X_{n, m}}$. Consequently, there is a natural isomorphism $\omega''\colon \Sigma \rightarrow \Sigma$ such that $({\rm Id}_\mathcal{T}, \omega'')$ is a triangle functor with $\delta'$ an isomorphism between triangle functors; see Lemma \ref{lem:iso-tri}.

We assume that $\omega''_{(X_{n, m})}=\Sigma(a_{n, m}+b_{n, m}\Delta_{n, m})$ for $a_{n, m}\neq 0$. We rotate the triangle (\ref{equ:tri1}) to get the following exact triangle
\begin{align}\label{equ:tri2}
 X_{n,m}\xrightarrow{\pi_{n, m}}  X_{n, m-1}\xrightarrow{c_{n, m-1, m-1}}  X_{m-1, m-1} \stackrel{t}\longrightarrow \Sigma(X_{n, m}),
\end{align}
where $t=-\Sigma(i_{n+1, m}\circ \cdots \circ i_{m-1, m}\circ i_{m, m})$. Applying the triangle functor $({\rm Id}_\mathcal{T}, \omega'')$ to this triangle, we have the following exact triangle
\begin{align*}
 X_{n,m}\xrightarrow{\pi_{n, m}}  X_{n, m-1}\xrightarrow{c_{n, m-1, m-1}}  X_{m-1, m-1} \xrightarrow{a_{n, m}t} \Sigma(X_{n, m}),
\end{align*}
where we use $\Sigma(\Delta_{n, m})\circ t=0$. We apply Proposition \ref{prop:almvan} to (\ref{equ:tri2}) to infer that $a_{n, m}=1$ for any $n\leq m$.

We define scalars $f_{n, m}$ for $n\leq m$ such that $f_{n, 0}=0$ and $b_{n, m}=f_{n-1, m-1}-f_{n, m}$. By Lemma \ref{lem:Ric} there is a natural isomorphism $\gamma\colon{\rm Id}_\mathcal{T}\rightarrow {\rm Id}_\mathcal{T}$ such that $\gamma_{(X_{n, m})}=\Sigma(1+f_{n, m}\Delta_{n, m})$.

We claim that $\gamma\colon ({\rm Id}_\mathcal{T}, \omega'')\rightarrow ({\rm Id}_\mathcal{T}, {\rm Id}_\Sigma)$ is an isomorphism of triangle functors. It suffices to prove that the right square in the following diagram commutes; compare Lemma \ref{lem:twofun}.
\[\xymatrix{
X_{n-1, m-1}\ar[d]^-{\gamma_{(X_{n-1, m-1})}} \ar[r]^-\phi & \Sigma(X_{n, m})\ar[d]^-{\gamma_{\Sigma(X_{n,m})}} \ar[rr]^-{\Sigma(1+b_{n, m}\Delta_{n, m})} && \Sigma(X_{n, m}) \ar[d]^-{\Sigma (\gamma_{(X_{n, m})})}\\
X_{n-1, m-1} \ar[r]^-\phi & \Sigma(X_{n, m}) \ar@{=}[rr] && \Sigma(X_{n, m}).
}\]
Here, $\phi\colon X_{n-1, m-1}\rightarrow \Sigma(X_{n, m})$ is an isomorphism of complexes whose $j$-th component $\phi^j$ equals $(-1)^j {\rm Id}_A$. The left square commutes by the naturality of $\gamma$. We observe that $\phi\circ \Delta_{n-1,m-1}\circ \phi^{-1}=\Sigma(\Delta_{n, m})$. Since we have $f_{n-1, m-1}=b_{n, m}+f_{n, m}$, it follows that the outer diagram commutes. Then we are done with the claim.

We summarise with the following composition of natural isomorphisms between triangle functors
$$(F, \omega)\stackrel{\delta}\longrightarrow (F', \omega')\stackrel{\delta'}\longrightarrow ({\rm Id}_\mathcal{T}, \omega'') \stackrel{\gamma}\longrightarrow ({\rm Id}_\mathcal{T}, {\rm Id}_\Sigma).$$
 This composition proves that $\mathcal{A}$ is $\mathbf{K}$-standard by Lemma \ref{lem:K-stan}.

The triangle center $Z_\vartriangle(\mathcal{T})$ is computed in \cite[Section 5]{KY}. It turns out that the homomorphism  (\ref{equ:CK}) for $\mathcal{A}$ is not an isomorphism; also see \cite[Lemma 3.2]{Ku}. By Lemma \ref{lem:stan-center},  $\mathcal{A}=A\mbox{-proj}$ is not strongly $\mathbf{K}$-standard. The second statement follows by Theorem \ref{thm:1}. \hfill $\square$

\subsection{Another example} Let $d\geq 2$.  Let $A$ be the algebra given by the following cyclic quiver
\[\xymatrix{
1\ar[r]^{\alpha_1} & 2\ar[r]^-{\alpha_2} & \cdots \ar[r] & d \ar@/^0.7pc/[lll]^-{\alpha_d}  }\]
with radical square zero. Then $A$ is a Nakayama Frobenius algebra. Denote by $e_s$ the primitive idempotent corresponding to the vertex $s$. Then the corresponding indecomposable projective $A$-module is $P_s=Ae_s=ke_s\oplus k\alpha_{s}$. Here, all the lower indices are viewed as elements in $\mathbb{Z}/d\mathbb{Z}$. For example, we identify $P_1$ with $P_{d+1}$.

In what follows, by the unnamed arrow $P_s\rightarrow P_{s-1}$, we mean the left $A$-module homomorphism induced by the multiplication of $\alpha_{s-1}$ from the right. More precisely, it sends $e_s$ to $\alpha_{s-1}$,  and $\alpha_s$ to $0$.

To describe the well-known structure of $\mathbf{K}^b(A\mbox{-proj})$, we introduce some notation; compare \cite[Section 5]{Bon}. For $s\in \mathbb{Z}/d\mathbb{Z}$ and $n\leq m$, we denote by $X_{s, n, m}$ the following complex of $A$-modules
$$\cdots \rightarrow 0 \rightarrow P_s \rightarrow P_{s-1}\rightarrow \cdots \rightarrow P_{s+n-m}\rightarrow 0\rightarrow \cdots, $$
where the nonzero components start at degree $n$ and end at degree $m$. In particular, we have $X_{s, n, n}=\Sigma^{-n}(P_s)$.

We denote by $i_{s, n, m}\colon X_{s, n, m}\rightarrow X_{s+1, n-1, m}$ the inclusion chain map,  and by $\pi_{s, n, m}\colon X_{s, n, m}\rightarrow X_{s, n, m-1}$ the projection if further $n<m$. For $n\leq m\leq l$, we have the following map $c_{s, n, m, l}\colon X_{s, n, m}\rightarrow X_{s+n-m-1,m, l}$
\[\xymatrix@C=15pt{
0\ar[r] & P_s \ar[r] & P_{s-1} \ar[r] & \cdots \ar[r] & P_{s+n-m} \ar[d] \ar[r] & 0 \\
 &  & & 0 \ar[r] & P_{s+n-m-1} \ar[r] & P_{s+n-m-2} \ar[r] & \cdots \ar[r] & P_{s+n-l-1} \ar[r] & 0.
}\]
Here, the only nonzero vertical map is induced by the multiplication of $\alpha_{s+n-m-1}$ from the right.

\begin{lem}\label{lem:fact2}
The following facts hold.
\begin{enumerate}
\item $\{X_{s, n, m}\; |\; s\in \mathbb{Z}/{d\mathbb{Z}}, \; n\leq m\}$ is a complete set of representatives of the isomorphism classes of indecomposable objects in $\mathbf{K}^b(A\mbox{-{\rm proj}})$.
\item For $n\leq m$ and $1\leq r$,  we have ${\rm Hom}(X_{s, n, m}, X_{s+r, n-r, m})=k(i_{s+r-1, n-r+1, m}\circ \cdots \circ i_{s+1, n-1, m} \circ i_{s, n, m})$.
    \item For $n<m$ and $1\leq r\leq m-n$, we have ${\rm Hom}(X_{s, n, m}, X_{s, n, m-r})=k(\pi_{s, n, m-r+1}\circ \cdots \circ \pi_{s, n, m-1}\circ  \pi_{s, n, m})$.
     \item For $n\leq m\leq l$, we have  ${\rm Hom}(X_{s, n, m}, X_{s+n-m-1, m, l})=kc_{s, n, m, l}$.
    \item For $n\leq m$, we have ${\rm Hom}(X_{s, n, m}, X_{s, n, m})=k{\rm Id}_{X_{s, n, m}}$.
    \item The morphisms $\{i_{s, n, m}, \pi_{s, n, m}, c_{s, n,m, l}\; |\; s\in \mathbb{Z}/{d\mathbb{Z}}, \; n\leq m\leq l\}$ span the $k$-linear category $\mathbf{K}^b(A\mbox{-{\rm proj}})$. \hfill $\square$
\end{enumerate}
\end{lem}

For $n<m$ and $s\in \mathbb{Z}/d\mathbb{Z}$, we observe the following exact triangle
\begin{align}\label{equ:tri3}
X_{s+n-m, m, m} \stackrel{t'}\longrightarrow X_{s, n, m}\xrightarrow{\pi_{s, n, m}} X_{s, n, m-1} \xrightarrow{c_{s, n, m-1, m-1}} X_{s+n-m, m-1, m-1},
\end{align}
where $t'=i_{s-1, n+1, m}\circ \cdots \circ i_{s+n-m+1, m-1, m}\circ i_{s+n-m, m, m}$.

\begin{lem}\label{lem:lambda2}
Let $(F, \omega)\colon \mathbf{K}^b(A\mbox{-{\rm proj}})\rightarrow \mathbf{K}^b(A\mbox{-{\rm proj}})$ be a pseudo-identity. Assume that $F(i_{s, n,m})=\lambda_{s, n, m}i_{s, n, m}$ and $F(\pi_{s, n, m})=\mu_{s, n, m}\pi_{s, n, m}$ for each $s\in \mathbb{Z}/{d\mathbb{Z}}$ and  $n\leq m$, where $\lambda_{s, n, m}$ and $\mu_{s, n, m}$ are nonzero scalars. Then the following statements hold.
\begin{enumerate}
\item For $n\leq m\leq l$, we have
\begin{align*}
F(c_{s, n,m, l})=(&\lambda_{s-1, n+1, m} \cdots \lambda_{s+n-m+1, m-1, m}\lambda_{s+n-m, m, m} \cdot \\
 &\mu_{s+n-m-1, m, m+1}\cdots \mu_{s+n-m-1, m, l-1}\mu_{s+n-m-1, m ,l})^{-1} c_{s, n, m, l}.\end{align*}
\item Assume for $s\in \mathbb{Z}/{d\mathbb{Z}}$ and $m\in \mathbb{Z}$ that $\omega_{(X_{s, m, m})}=a_{s, m} {\rm Id}_{\Sigma(X_{s, m, m})}$ with $a_{s, m}$ nonzero scalar. Then we have $a_{s, m}=a_{s', m}$ for any $s, s'$. This common value is denoted by $a_m$.
\item    For $n<m$,  we have
    \begin{align}\label{equ:lambda2}
    &\lambda_{s-1, n+1, m}\cdots \lambda_{s+n-m+1, m-1, m}\lambda_{s+n-m, m, m}\mu_{s, n, m}a_{m}\\
     =&\lambda_{s-1, n+1, m-1}\cdots \lambda_{s+n-m+2, m-2, m-1}\lambda_{s+n-m+1, m-1, m-1}.\nonumber
    \end{align}
\end{enumerate}
\end{lem}

\begin{proof}
The proof is similar to the one of Lemma \ref{lem:lambda}. Denote by $\phi_s\colon P_s\rightarrow P_{s-1}$  the above unnamed arrow. Recall that $F\Sigma^m(\phi_s)=\Sigma^m(\phi_s)$ for each $m\in \mathbb{Z}$ and $s\in \mathbb{Z}/{d\mathbb{Z}}$.  Then we obtain (1). For (2), we  apply the naturality of $\omega$ to the morphism $\Sigma^{-m}(\phi_s)\colon X_{s, m, m}\rightarrow X_{s-1, m, m}$ and obtain $a_{s, m}=a_{s-1, m}$. For (3), it suffices to apply $(F, \omega)$ to the triangle (\ref{equ:tri3}). We omit the details.
\end{proof}

The following result is analogous to Theorem \ref{thm:exam1}.

\begin{prop}
Let $A$ be the above algebra given by a cyclic quiver with radical square zero.  Then $A\mbox{-{\rm proj}}$ is strongly $\mathbf{K}$-standard, and thus $A\mbox{-{\rm mod}}$ is strongly $\mathbf{D}$-standard.
\end{prop}

\begin{proof}
The proof is similar to the one of Theorem \ref{thm:exam1}; indeed, it is much easier. We only give a sketch. Set $\mathcal{A}=A\mbox{-{\rm proj}}$ and $\mathcal{T}=\mathbf{K}^b(A\mbox{-{\rm proj}})$.

Let $(F, \omega)$ be a pseudo-identity on $\mathcal{T}$. We assume that $F(i_{s, n,m})=\lambda_{s, n, m}i_{s, n, m}$ and $F(\pi_{s, n, m})=\mu_{s, n, m}\pi_{s, n, m}$ for some nonzero scalars $\lambda_{s, n, m}$ and $\mu_{s, n, m}$. Assume that $\omega_{(X_{s, m, m})}=a_{m} {\rm Id}_{\Sigma(X_{s, m, m})}$ for nonzero scalars $a_{m}$; see Lemma \ref{lem:lambda2}(2). We choose nonzero scalars $c_{m}$ such that $c_{0}=1$ and $a_{m}=c_{m-1}(c_{m})^{-1}$ for each $m\in \mathbb{Z}$. For each complex $X$, we fix an isomorphism
$$\delta_X\colon F(X)=X\longrightarrow X=F'(X)$$
 such that $\delta_{\Sigma^m(X)}=c_{-m}{\rm Id}_{\Sigma^m(X)}$ for any $X\in \mathcal{A}$ and $m\in \mathbb{Z}$,  and that
$$\delta_{(X_{s, n, m})}=c_{m} (\lambda_{s-1, n+1, m} \cdots \lambda_{s+n-m+1, m-1, m}\lambda_{s+n-m, m, m})^{-1}{\rm Id}_{X_{s, n, m}}$$
for $n<m$. As required, we have  $\delta_{(X_{s, m, m})}=c_{m}{\rm Id}_{X_{s, m, m}}$.

We use these isomorphisms $\delta_X$'s as the adjusting isomorphisms to obtain a new triangle functor $(F', \omega')$, which is still a pseudo-identity. It follows that $F'(i_{s, n,m})=i_{s, n, m}$ and $F'(\pi_{s, n, m})=\pi_{s, n, m}$. Here, the latter identity relies on (\ref{equ:lambda2}). By Lemma \ref{lem:lambda2}(1), we have $F'(c_{s, n, m, l})=c_{s, n, m, l}$. It follows from Lemmas \ref{lem:fact2}(6) and \ref{lem:iso-gen} that  there is a natural isomorphism $\delta'\colon (F', \omega')\rightarrow ({\rm Id}_\mathcal{T}, \omega'')$ of triangle functors.

We observe that $\mathcal{T}$ is a non-degenerate block. We apply Lemma \ref{lem:fact2}(5) and Proposition \ref{prop:ZZ} to infer that $Z_\vartriangle(\mathcal{T})=k$ and that $\omega''={\rm Id}_\Sigma$. By Lemmas \ref{lem:K-stan} and \ref{lem:stan-center},  we infer that $\mathcal{A}=A\mbox{-proj}$ is strongly $\mathbf{K}$-standard. The second statement follows from Theorem \ref{thm:1}.
\end{proof}

\vskip 5pt

\noindent {\bf Acknowledgements}\quad We thank Wei Hu,  Zengqiang Lin and David Ploog for their helpful comments and encouragement. The work is supported by National Natural Science Foundation of China (No.s 11431010, 11522113 and 11571329).

\bibliography{}

\vskip 10pt

 {\footnotesize \noindent Xiao-Wu Chen, Yu Ye\\
  School of Mathematical Sciences, University of Science and Technology of
China, Hefei 230026, Anhui, PR China \\
Wu Wen-Tsun Key Laboratory of Mathematics, USTC, Chinese Academy of Sciences, Hefei 230026, Anhui, PR China.\\
URL: http://home.ustc.edu.cn/$^\sim$xwchen, http://staff.ustc.edu.cn/$^\sim$yeyu.}

\end{document}